\documentclass{article}

\usepackage{amsfonts, enumerate, amsmath, mathtools, amsthm, amssymb, xcolor, cite, tikz, verbatim, graphicx, soul, float, tikz-network, subcaption, hhline,authblk}

\usepackage[margin=1in]{geometry}

\usepackage[pdftex,hypertexnames=false,linktocpage=true]{hyperref}
\hypersetup{colorlinks=true,linkcolor=black,anchorcolor=blue,citecolor=black,filecolor=blue,urlcolor=blue,bookmarksnumbered=true,pdfview=FitB}

\usepackage{lineno}
\numberwithin{equation}{section}

\makeatletter
\newcommand*{\centerfloat}{%
	\parindent \z@
	\leftskip \z@ \@plus 1fil \@minus \textwidth
	\rightskip\leftskip
	\parfillskip \z@skip}
\makeatother

\newtheorem{theorem}{Theorem}[section]
\newtheorem{corollary}[theorem]{Corollary}
\newtheorem{lemma}[theorem]{Lemma}
\newtheorem{prop}[theorem]{Proposition}
\newtheorem{conjecture}[theorem]{Conjecture}
\newtheorem{observation}[theorem]{Observation}
\newtheorem{question}[theorem]{Question}

\theoremstyle{definition}
\newtheorem{definition}[theorem]{Definition}
\newtheorem{remark}[theorem]{Remark}
\newtheorem{example}[theorem]{Example}

\newcommand\set[1]{\left\{ #1 \right\}}
\newcommand\abs[1]{\left| #1 \right|}
\newcommand\parens[1]{\left( #1 \right)}

\newcommand\N{\mathbb{N}}

\newcommand\A{\mathcal{A}}
\newcommand\W{\mathcal{W}}
\newcommand{\diam}{\mathord{\diamond}}
\newcommand{\bull}{\mathord{\bullet}}

\definecolor{grey}{gray}{.75}
\definecolor{0}{RGB}{250,0,0}
\definecolor{1}{RGB}{0,10,255}
\definecolor{2}{RGB}{220,190,0}
\definecolor{3}{RGB}{0,170,130}
\definecolor{$\diam$}{RGB}{255,255,255}

\newcommand{\directionarrow}{\draw[ultra thick, grey,->] (0,0) +(75:.5) arc(75:-255:.5);}

\title{The Existence and Structure of Universal Partial Cycles}
\author{Dylan Fillmore, Bennet Goeckner, Rachel Kirsch, \\ Kirin Martin, and Daniel McGinnis}

\begin{document}
	
	\maketitle
	
	\abstract{A universal partial cycle (or upcycle) for $\A^n$ is a cyclic sequence that covers each word of length $n$ over the alphabet $\A$ exactly once---like a De Bruijn cycle, except that we also allow a wildcard symbol $\diam$ that can represent any letter of $\A$.  Chen et al.~in 2017 and Goeckner et al.~in 2018 showed that the existence and structure of upcycles are highly constrained, unlike those of De Bruijn cycles, which exist for every alphabet size and word length. Moreover, it was not known whether any upcycles existed for $n \ge 5$. We present several examples of upcycles over both binary and non-binary alphabets for $n = 8$. We generalize two graph-theoretic representations of De Bruijn cycles to upcycles. We then introduce novel approaches to constructing new upcycles from old ones. Notably, given any upcycle for an alphabet of size $a$, we show how to construct an upcycle for an alphabet of size $ak$ for any $k \in \N$, so each example generates an infinite family of upcycles. We also define folds and lifts of upcycles, which relate upcycles with differing densities of $\diam$ characters. In particular, we show that every upcycle lifts to a De Bruijn cycle. Our constructions rely on a different generalization of De Bruijn cycles known as perfect necklaces, and we introduce several new examples of perfect necklaces. We extend the definitions of certain pseudorandomness properties to partial words and determine which are satisfied by all upcycles, then draw a conclusion about linear feedback shift registers. Finally, we prove new nonexistence results based on the word length $n$, alphabet size, and $\diam$ density.}

	\section{Introduction}

	Let $\A=\{0,\dots,a-1\}$ for some $a\ge2$. Let $n$ be a positive integer, and let $\A^n$ denote the set of words of length $n$ with symbols in $\A$. A \emph{De Bruijn cycle} for $\A^n$ is a cyclic sequence of symbols $(u_1\cdots u_{N})$ such that for every $w\in \A^n$ there is precisely one index $i$ for which $u_iu_{i+1}\cdots u_{i+n-1}=w$. Here indices are taken modulo $N$, the \emph{length} of the De Bruijn cycle, which is always $a^n$. For example, $(00010111)$ is a De Bruijn cycle for $\set{0,1}^3$ because its consecutive substrings of length $3$ — $000$, $001$, $010$, $101$, $011$, $111$, $110$, and $100$ — cover each binary word of length $3$ exactly once. De Bruijn cycles exist for every alphabet size $a$ and every subword length $n$.
	
	In this paper, we are primarily concerned with universal partial cycles (which we call ``upcycles'' for brevity). A \emph{universal partial cycle} (or \emph{upcycle}) covers each word of $\A^n$ exactly once even more compactly than a De Bruijn cycle by using a wildcard symbol $\diam$ which can represent any letter of $\A$. For example, $(001\diam 110 \diam)$ is a universal partial cycle for $\{0,1\}^4$ because its consecutive substrings of length $4$ cover words $0010$ and $0011$, $0101$ and $0111$, $1011$ and $1111$, $0110$ and $1110$, $1100$ and $1101$, $1000$ and $1010$, $0000$ and $0100$, and finally $0001$ and $1001$. Thus, for every $w \in \set{0,1}^4$, there is exactly one index $i$ for which $a_ia_{i+1}\cdots a_{i+n-1}$ covers $w$, where indices are taken modulo $8$.
	
	The study of upcycles stems from the study of \emph{universal partial words} (or \emph{upwords}), first introduced in \cite{CKMS17} in 2017. Upwords are similar to upcycles but are not read cyclically. In \cite{CKMS17}, Chen et al. proved that upwords for $\set{0,1}^n$ exist for every $n \ge 2$. For example, $\diam^{n-1}01^n$ is an upword for $\set{0,1}^n$ because each $w \in \set{0,1}^n$ is covered only by $a_ia_{i+1}\cdots a_{i+n-1}$, where $i$ is one more than the number of consecutive $1$'s at the end of $w$. They also proved structural constraints on and nonexistence results for upcycles (called ``cyclic upwords'' in \cite{CKMS17}) and noted that $(001\diam 110 \diam)$ was the only known binary upcycle. Several years earlier, in \cite{BSSW10}, Blanchet-Sadri et al.~introduced \emph{De Bruijn partial words}, partial words of minimum length that cover every word in $\A^n$. De Bruijn partial words thus exist for every $\A$ and $n$ but may cover a given word of length $n$ more than once.
	
	De Bruijn cycles have been studied extensively and have myriad applications, including genome assembly \cite{CPT11}, neuroscience \cite{neural11}, and robotics \cite{GS16}. De Bruijn sequences with good autocorrelation properties (see Section~\ref{sec:pseudo}) are applied throughout communications and cryptography \cite[Chapter~5]{GG05}. Universal partial words are comparatively quite new, but variants already have been studied in computer science \cite{BF20} and computational biology \cite{OPKFB17}. More recently, upcycles have been used to construct two-dimensional analogues called uptori \cite{CKKP25}.
	
	It was shown in \cite{G18} that every upword over a non-binary alphabet is cyclic, meaning that its first and last $n-1$ characters can be glued together to form a cycle that covers each word of length $n$ exactly once. Thus the study of upcycles encompasses the study of all upwords over non-binary alphabets. In the same paper, an example of an upcycle for $\A^4$ for any even $|\A|$ was constructed, and the authors asked if there exist upcycles for $\A^n$ with $n\geq 5$ (and some $\A$). In Section~\ref{sec:newbinary}, we present the first known upcycles for $n\ge 5$, specifically for $n=8$ and $a=2$. 
	
	In Section~\ref{sec:constructions}, we introduce some new ways to create perfect necklaces. In Section~\ref{sec:alph_mult}, we show that given an upcycle for a word of length $n$ over an alphabet of size $a$, we can construct an upcycle for word length $n$ over an alphabet of size $ak$ for any $k\in \N$. Thus, we deduce the existence of upcycles for $\A^8$ whenever $|\A|$ is even. Then, in Section~\ref{sec:lift}, we show how upcycles with greater diamond density can be ``unfolded'' into upcycles with certain lower diamond densities; in particular we show that any upcycle can be unfolded into a De Bruijn cycle. We also introduce two new representations of upcycles as substructures of De Bruijn graphs. In Section~\ref{sec:pseudo} we show that upcycles satisfy some pseudorandomness properties, further distinguishing them from binary upwords and highlighting their similarities with De Bruijn cycles. At the same time, we show that unfolding an upcycle with diamondicity 1 never produces a De Bruijn cycle with ``good'' autocorrelation. In Section~\ref{sec:nonexist}, we give further nonexistence conditions on upcycles, narrowing the search space for more upcycles. We conclude with open questions in Section~\ref{sec:open}.

	\section{Preliminaries}
	
	An \emph{alphabet} $\A$ is a set of symbols, which we call \emph{letters}. A \emph{character} refers to a letter or a diamond $\diam$ (which is assumed to not be a letter of $\A$). Throughout the paper, the alphabet size $|\A|\ge2$ is denoted by $a$, and we assume $\A=\{0,\dots,a-1\}$ unless otherwise specified. A \emph{word} (or \emph{total word}) over $\A$ is a sequence of letters from $\A$, denoted by $u=u_1\cdots u_N$, and a \emph{partial word} is a sequence of characters from $\A\cup \{\diam\}$. A \emph{cyclic word} over $\A$ is a cyclically ordered sequence of letters from $\A$, denoted by $u=(u_1\cdots u_N)$, and a \emph{cyclic partial word} is defined analogously. The $i$th character of a (cyclic) partial word $x$ is denoted by $x_i$. Indices of cyclic partial words are considered modulo the length of the cyclic partial word (using indices $1$ to $N$ for notational convenience). A \emph{substring} of a (cyclic) partial word $x$ is a partial word $x_i\cdots x_{i+k-1}$ for some $i$ and $k$. We sometimes refer to $x_i \cdots x_{i+k-1}$ as the substring at position $i$. We use the standard notation $uv$ for the concatenation of partial words $u$ and $v$, and $v^q$ for the concatenation of $q$ copies of $v$.
	
	For a (cyclic) partial word $x$ and a partial word $y$, we say that $x$ \emph{covers} $y$ (or $y$ is \emph{covered} by $x$) if $x$ has a substring $x'$ of the same length as $y$ such that either $x'_i =\diam$ or $x'_i = y_i$ for all $i$. A \emph{universal partial cycle} (or \emph{upcycle}) is a cyclic partial word that covers each word of $\A^n$ exactly once. A \emph{window} of an upcycle $u=(u_1\cdots u_N)$ for $\A^n$ is a substring of $u$ of length $n$, and the \emph{$k$-window at position $i$} of a (cyclic) partial word $u$ is $u_iu_{i+1}\cdots u_{i+k-1}$ (note that a window is an $n$-window when $u$ is understood to be an upcycle for $\A^n$). Thus each word of $\A^n$ is covered by exactly one window of an upcycle. We say that an upcycle is \emph{non-trivial} if it contains at least one diamond and at least one non-diamond character. This rules out De Bruijn cycles and the degenerate upcycle $(\diam)$. The following proposition shows that for any nontrivial upcycle, the length of the upcycle exceeds $n$.
	
	\begin{prop}\label{prop:upcycle_length}
		If $u$ is a nontrivial upcycle for $\A^n$, then $|u| > n$.
	\end{prop}
	\begin{proof}
		Suppose $u$ is a nontrivial upcycle of length at most $n$. Recall $\abs{\A} \ge 2$, so we assume $0, 1 \in \A$. The upcycle is nontrivial, so without loss of generality $u$ contains the letter $0$. Therefore no window of $u$ covers $1^n$, contradicting that $u$ is an upcycle.
	\end{proof}
	
	By \cite[Lemma~14]{CKMS17}, for any upcycle $u$ for $\A^n$, if $u_i = \diam$ then $u_{i+n} = \diam$, i.e., the diamond characters of $u$ are $n$-periodic. Thus every window of an upcycle contains the same number of $\diam$ characters, and we call this number the \emph{diamondicity} of the upcycle, denoted by $d$. This implies that the length of any upcycle is $a^{n-d}$, and the number of diamond characters is $(d/n)a^{n-d}$. We sometimes refer to an upcycle for word length $n$ over an alphabet of size $a$ with diamondicity $d$ as an $(a,n,d)$ upcycle. The only known non-trivial upcycles at the time of writing have $d=1$.
	
	Given a (cyclic) partial word $u$, its \emph{frame} is the (cyclic) binary total word over the alphabet $\{\bull,\diam\}$ obtained by replacing every letter (non-diamond character) of $u$ with a $\bull$ symbol. A \emph{window frame} is the frame of a window. For example, given the upcycle $u=(001\diam 110\diam)$ for $a=2$ and $n=4,$ the frame of $u$ is $(\bull \bull \bull \diam \bull \bull \bull \diam)$. The second window of $u$ is $01\diam 1$, and the corresponding window frame is $\bull \bull \diam \bull$.
	
	A \emph{universal partial word} (or \emph{upword}) for $\A^n$ is a partial word $u=u_1\dots u_N$ that covers every word in $\A^n$ exactly once when read linearly (and not cyclically). Both $\diam \diam 0111$ and $0\diam 011100$ are upwords for $\set{0,1}^3$. Diamondicity is not well-defined for upwords in general. An upword $u$ of the form $u=vwv$ for some partial words $v,w$ with $\abs{v}=n-1$ can be made into an upcycle $(vw)$. Conversely, any upcycle can be made into $a^{n-d}$ different upwords by splitting the cycle at any position. For example, the upcycle $(001\diam110\diam)$ produces the upwords $001\diam110\diam001$ and $0\diam001\diam110\diam0$, among others. As mentioned in the introduction, all upwords over non-binary alphabets can be expressed as upcycles and thus have well-defined diamondicity. (We note in passing that Proposition~\ref{prop:upcycle_length} gives a shorter way to prove \cite[Theorem~4.8]{G18} which says that every non-binary upword corresponds to an upcycle.)
	
	We note two straightforward operations that can be performed on an upcycle (or upword) to produce another upcycle. If 
	\(
	u=(u_1\dots u_N)
	\)
	is an upcycle and $\pi$ is a permutation of the letters of the alphabet and $\pi(\diam) = \diam$, then $\pi(u)=(\pi(u_1)\dots \pi(u_N))$ is an upcycle. We say that $\pi(u)$ is obtained from $u$ by permuting the letters of the alphabet. Also, $(u_N\cdots u_1)$ is an upcycle which we call the reversal of $u$, or we say that this upcycle is obtained by reversing $u$.
	
	For a positive integer $n$, we use the notation $[n]$ for the set $\set{1,2, \ldots, n}$.
	
	\subsection{Perfect necklaces, astute graphs, and De Bruijn graphs}\label{sec:astute_graphs_etc}
	
	In Sections~\ref{sec:alph_mult} and \ref{sec:lift} we introduce some constructions to produce new upcycles from a given upcycle. Two of these constructions require use of the following notion (see, e.g., \cite{AB16}).
	
	\begin{definition}
		An \emph{$(a,n,t)$-perfect necklace} $v$ is a cyclic word of length $t\!\cdot\!a^n$ such that, for every $j\in [t]$, each word of $\set{0,\ldots, a-1}^n$ is the $n$-window at position $i$ of $v$ for a unique $i \in [ta^n]$ with $i \equiv j \pmod{t}$.
	\end{definition}
	
	By definition, De Bruijn cycles for $\set{0,1,\dots,a-1}^n$ are exactly $(a, n,1)$-perfect necklaces.
	
	\begin{example}\label{ex:perf_neck}
		Let $\A = \set{0,1}$. Consider the following cyclic words.
		\begin{center}
			\begin{tabular}{ccc|rl}
				$a$ & $n$ & $t$ & & $\hspace{-1cm}(a,n,t)$-perfect necklace $\phantom{\Big|}$\\\hline
				$2$ & $2$ & $2$ & $u =$ & $\hspace{-.32cm}(00 ~ 11 ~ 10 ~ 01)$ \\
				$2$ & $1$ & $4$ & $v =$ & $\hspace{-.32cm}(00 10 ~ 11 01)$ \\
				$2$ & $3$ & $3$ & $w =$ & $\hspace{-.32cm}(000 ~ 001 ~ 010 ~ 011 ~ 100 ~ 101 ~ 110 ~ 111)$ \\
			\end{tabular}
		\end{center}
		Then $u$ is a $(2, 2,2)$-perfect necklace, $v$ is a $(2, 1,4)$-perfect necklace, and $w$ is a $(2, 3,3)$-perfect necklace.
	\end{example}
	
	The existence of perfect necklaces will be important for some results in this paper, and the following graph allows one to deduce the existence of perfect necklaces.
	
	\begin{definition}[Astute graph and De Bruijn graph]\label{def:AstuteBad}
		Let $a,n,t \in \mathbb{N}$ and let $\A = \set{0,\dots,a-1}.$ The \emph{astute graph} $G(a,n,t)$ is the directed graph with vertices $(u,s)$ where $u \in \A^{n}$ and $s \in \set{0,1,\dots,t-1}$.
		
		There is an edge pointing from $(u,s)$ to $(u',s')$ if and only if both
		\begin{enumerate}
			\item $u=xv$ and $u'=vy$ for some $x,y \in\A$ and $v\in \A^{n-1}$, and
			\item $s+1 \equiv s' \pmod t$.
		\end{enumerate}
		It is sometimes useful to label the edges by $y$ and other times useful to label the edges by $xvy$.
		
		The \emph{De Bruijn graph} $B(a,n)$ is isomorphic to the astute graph $G(a,n,1)$. Its vertices are the words $u \in \A^n$, and its edges are labeled by either $y$ or $xvy$, as in the astute graph.
	\end{definition}

	The astute graph $G(a,n,t)$ can also be viewed as the tensor (also called direct or Kronecker) product of the De Bruijn graph $B(a,n)$ with a directed cycle on $t$ vertices. We note the similarity between the astute graph $G(2,n,t)$ and the Praeger--Xu graph $PX(n,t)$ as mentioned in, for example, \cite{JP22}.
	
	Astute graphs are strongly connected with $\deg^-(v) = \deg^+(v) = a$ for every vertex $v$, so they are Eulerian. Furthermore, an Euler tour of $G(a,n-1,t)$ induces an $(a,n,t)$-perfect necklace for $\A$ \cite[Section~3]{AB16}, which immediately gives the following result.
	
	\begin{prop}[Section~3 in \cite{AB16}]\label{prop:PerfNecksExist!}
		There exists an $(a,n,t)$-perfect necklace for any $a,n,t\in \mathbb{N}$. 
	\end{prop}
	
	The following construction of perfect necklaces is particularly nice and useful.
	
	\begin{lemma}[Corollary~6 in \cite{AB16}]\label{lemma:LexOrderPerfNeck}
		The concatenation of the words of $\set{0,1,\dots,a-1}^n$ in lexicographic order gives an $(a,n,n)$-perfect necklace.
	\end{lemma}

	\section{New Binary Upcycles}\label{sec:newbinary}
	
	For $n\le 3$, there are no upcycles for any alphabet \cite[Proposition~5.1]{G18}. For $n=4$, up to symmetries the only binary upcycle is $u=(001\diam110\diam)$, which we prove in detail in Example~\ref{ex:a=2,n=4,uniqueupcycle}. There were previously no known upcycles for $n\ge 5$. By \cite[Theorem~4.11]{G18}, the only possible upcycles for binary alphabets with diamondicity $d=1$ have word length $n=2^k$. 
	
	\subsection{List of binary upcycles for $n=8$}\label{sec:list}
	We used a computer search to obtain the following new binary upcycles.
	
	\begin{theorem}\label{n=8d=1}
		Binary upcycles for $n=8$ exist.
	\end{theorem}
	\begin{proof}
		Each of the following cyclic partial words is an upcycle for $\set{0,1}^8$.
		\begin{enumerate}
			\item $\begin{aligned}[t](&0000010\diam1111101\diam0010010\diam1101101\diam1110000\diam0001111\diam1110011\diam0101100\diam\\&1110010\diam0101001\diam1000110\diam0100001\diam1011110\diam0101101\diam0000110\diam1101001\diam)\end{aligned}$
			\item $\begin{aligned}[t](&0000001\diam1111110\diam0111001\diam1110110\diam0100001\diam1011110\diam0010001\diam1101110\diam\\&0101001\diam1100100\diam0011011\diam1100110\diam0110100\diam1001010\diam0110000\diam1001110\diam)\end{aligned}$
			\item $\begin{aligned}[t] (&0000001\diam1101110\diam1111001\diam0101110\diam0010101\diam1101010\diam0010001\diam0001110\diam\\&1011001\diam0000110\diam1110001\diam0100111\diam1011000\diam0100110\diam1010001\diam1111110\diam)\end{aligned}$
			\item $\begin{aligned}[t](&0000001\diam0111110\diam1010001\diam0100110\diam1011001\diam0101110\diam0010101\diam1101010\diam\\&0010001\diam1001110\diam0110001\diam1101110\diam1000011\diam0111100\diam1000001\diam1111110\diam)\end{aligned}$
			\item $\begin{aligned}[t](&0100001\diam1011110\diam0101101\diam1110011\diam0101100\diam1110010\diam0101001\diam1000110\diam\\&0100000\diam1011101\diam0010010\diam1111101\diam0110000\diam0001111\diam1110000\diam1001101\diam)\end{aligned}$
			\item $\begin{aligned}[t](&0000001\diam0111101\diam1010010\diam1101101\diam0010110\diam1100001\diam1010110\diam0100001\diam\\&0010010\diam0111001\diam1000110\diam0111100\diam1010011\diam0111110\diam1000001\diam1111110\diam)\end{aligned}$
			\item $\begin{aligned}[t](&1011010\diam1110011\diam0100000\diam0011111\diam0100100\diam1011011\diam0101100\diam1110111\diam\\
				&0101000\diam1110001\diam0001010\diam1100000\diam1011111\diam1100001\diam0011010\diam0100101\diam)\end{aligned}$
		\end{enumerate}
	\end{proof}
	
	The computer search was not exhaustive, and we do not know how many distinct upcycles there are for $\set{0,1}^8$. After the preparation of this paper, William D. Carey improved our computer search and found hundreds of thousands of upcycles for $\set{0,1}^8$, distinct up to rotation \cite{Carey}. 
	In Corollary~\ref{cor:alphamult} we will see that each of these new binary upcycles generates upcycles for $\A^8$ for every even-sized alphabet $\A$.
	
	\subsection{Distinctness and the cross-join operation for upwords}\label{sec:chop'n'swap}
	
	In this section, we confirm that the upcycles presented in Theorem \ref{n=8d=1} are distinct up to symmetries (rotation, complementation, and reversal) and up to a new partial-words version of the cross-join operation, which generalizes the cross-join operation for De Bruijn cycles (see e.g. \cite{Fredricksen75,Fredricksen82}) to apply to upwords. This operation modifies an upword for $\A^n$ to obtain a new upword for $\A^n$. 
	
	If $u$ and $v$ are De Bruijn cycles for $\set{0,1}^n$, then $u$ can be obtained from $v$ by a sequence of cross-join operations \cite{MS15,JL20}. In contrast, we argue that there are no two upcycles presented in Theorem~\ref{n=8d=1} for which one can be obtained from the other by applying a sequence of operations: rotation, complementation, reversal, or cross-join operations. 
	
	First we define this operation for upwords. For De Bruijn cycles the graph-theoretic representation of this operation is straightforward, but for upcycles the analogous structures would have multiple strands (see Section \ref{subsec:Hamilton}). Here it is simpler to view the upcycle as a cyclic total word over an extended alphabet.
	
	\begin{prop}\label{thm:cns}
		Let $x$ and $y$ be partial words over $\A$ of length $n-1$. Suppose $u$ is an upword of which $x$ is a substring starting at indices $i_x$ and $j_x$, and $y$ is a substring starting at indices $i_y$ and $j_y$, such that $i_x < i_y < j_x < j_y$ (in the cyclic ordering of the indices of $u$). Then 
		\[
		u' := u_1 \cdots u_{i_x-1} u_{j_x} \cdots u_{j_y-1} u_{i_y} \cdots u_{j_x-1}u_{i_x} \cdots u_{i_y-1}u_{j_y} \cdots u_{\abs{u}}\] is an upword containing the same sets of windows (length-$n$ subwords over $\A\cup\{\diam\}$).
	\end{prop}
	
	\begin{proof}
		To prove that $u'$ is an upword for $\A^n$, it suffices to show that the windows of $u'$ are the windows of $u$. Observe that any window of $u'$ contained entirely within $u_1 \cdots u_{i_x-1}$, $u_{j_x} \cdots u_{j_y-1}$, $u_{i_y} \cdots u_{j_x-1}$, $u_{i_x} \cdots u_{i_y-1}$ or $u_{j_y} \cdots u_{\abs{u}}$ is unchanged. All other windows contain at least one of the pairs $u_{i_x-1} u_{j_x}$, $u_{j_y-1} u_{i_y}$, $u_{j_x-1}u_{i_x}$, and $u_{i_y-1}u_{j_y}$. Suppose $u_{i_x-1} u_{j_x}$ is the first of these pairs contained in a particular window; the other cases are similar. Because $u_{i_x-1} u_{j_x}$ is the first such pair, the part of the window before it (including $u_{i_x-1}$) is contained in $u_1 \cdots u_{i_x-1}$, which is unchanged. The part of the window after this pair (including $u_{j_x}$) has length at most $n-1$, so is contained in $u_{j_x} \cdots u_{j_x + n-2} = u_{i_x} \cdots u_{i_x + n-2}$. Thus the whole window is unchanged.
	\end{proof}
	
	A cross-join operation on an upcycle can be similarly defined, and a similar proof holds. An example of an upcycle $u$ for $\{0,1,2,3\}^4$ was given in Example 5.3 of \cite{G18}. In Example~\ref{ex:cns} below we demonstrate an upcycle $u'$ obtained from $u$ by a cross-join. We also show that $u'$ cannot be obtained from $u$ by permuting the letters and/or reversal.

	\begin{example}\label{ex:cns}
		Assume $\A = \set{0,1,2,3}$ and $n=4.$ Let $x=3\diam1$ and $y= 21\diam$. The upcycle given in Example 5.3 of \cite{G18} is
		$$
		u = (001\diam110\diam00\underline{3\diam112\diam021\diam}130\diam02\underline{3\diam132\diam201\diam310\diam203\diam312\diam221\diam}330\diam223\diam332\diam)
		$$
		where the underlined substrings both begin with $x$ and end with $y$. Performing the upcycle cross-join operation from Proposition~\ref{thm:cns} to $u$ with this $x$ and $y$, we get
		$$
		u' = (001\diam110\diam00\underline{3\diam132\diam201\diam310\diam203\diam312\diam221\diam}130\diam02\underline{3\diam112\diam021\diam}330\diam223\diam332\diam)
		$$
		with the substrings beginning with $x$ and ending with $y$ again underlined for emphasis.
		
		To see that $u'$ cannot be obtained from $u$ by permutation of letters and/or reversal, note that in both upwords, the word $11110000$ is covered by $1\diam 110\diam 00$, and no other word of the form $aaaabbbb$ is covered. If $u'$ could be obtained from $u$ by permuting the letters and/or reversal, then it would have to be obtained by applying a permutation that transposes $0$ and $1$ and reversing. But the window covering $11110000$ in both $u'$ and $u$ is $1\diam110\diam00$, and applying a permutation that transposes $0$ and $1$ and reversing $u'$ will result in an upword that covers $11110000$ with the window $11\diam100\diam0$.
	\end{example}

	\begin{prop}
		There are no two upcycles presented in Theorem~\ref{n=8d=1} for which one can be obtained from the other by applying a sequence of the operations: permuting the letters of the alphabet, reversal, or cross-join operations. 
	\end{prop}
	\begin{proof}
		Given an upcycle $u=(u_1\dots u_N)$, let $\W_u$ be the multiset of subwords of $u$ that contain no diamonds and are bordered by diamonds in $u$. In other words, $w\in \W_u$ if and only if $w=u_i u_{i+1}\dots u_k$ has no diamonds and $u_{i-1}=u_{k+1}=\diam$, where indices are taken modulo $N$.
		
		By Proposition~\ref{thm:cns}, if $u'$ is obtained from $u$ by applying cross-join operations, then every window of $u$ is a window of $u'$. This implies that $\W_u=\W_{u'}$ since every element of $\W_u$ is a subword of a window of $u$ and the same is true of $u'$.

		By inspection, for each $u\neq u'$ presented in Theorem~\ref{n=8d=1}, we have that $\W_u\neq \W_{u'}$, $\W_u$ is not equal to the set obtained by reversing each element in $\W_{u'}$, $\W_u$ is not obtained by applying the permutation $0\mapsto 1$ and $1\mapsto 0$ to each element of $\W_{u'}$. Moreover, $\W_u$ is not obtained by applying the permutation $0\mapsto 1$ and $1\mapsto 0$ to each element of $\W_{u'}$ and then reversing. It then follows that $u$ cannot be obtained from $u'$ by the above mentioned operations.
	\end{proof}
	
	\section{Constructions of perfect necklaces}\label{sec:constructions}
	
	Given the important role of perfect necklaces throughout this paper, we present three constructions of perfect necklaces in Propositions \ref{prop:pn_const1} and \ref{prop:pn_const2}. These constructions modify a given De Bruijn sequence to build the perfect necklace, and this construction method may be of independent interest.
	
	The proposition below describes a way to produce larger perfect necklaces given some fixed perfect necklace. We note that the case $r=0$ is proven in \cite[Proposition~13]{AB16}. We use it in the proof of Theorem~\ref{thm:divalphmult}.
	
	\begin{prop}\label{prop: smallperfect}
		Let $0 \leq r < n$. Let $w = (w[1] \cdots  w[a^n])$ be an $(a,n,n+r)$-perfect necklace, where each $w[p]$ is a word of length $n+r$. Then for every $q \geq 1$, $v = (v[1] \cdots  v[a^n])$ is an $(a,n,nq + r)$-perfect necklace, where $v[p] = (w[p]_1 \cdots w[p]_n)^q w[p]_{n+1} \cdots w[p]_{n+r}$.
	\end{prop}
	
	\begin{proof}
		We show that for all $i\in\set{1,\ldots,nq+r}$, the $n$-window starting at the $i$th position of $v[p]$ is distinct across all $p\in\set{1,\ldots,a^n}$. Since each $v[p]$ has length $nq+r$, this implies that $v=(v[1] \cdots  v[a^n])$ is an $(a,n,nq + r)$-perfect necklace.
		
		Since the first $nq$ characters of each $v[p]$ are $n$-periodic, we first consider $i\in\set{1,\ldots,nq-n+1}$. Let $\ell$ be the unique integer such that $i=n\ell+j$ for some $j\in\set{1,\ldots,n}$. Then for any $p\in\set{1,\ldots,a^n}$,
		\begin{align*}v[p]_i \cdots v[p]_{i+n-1} &= v[p]_{n\ell+j} \cdots v[p]_{n\ell+j+n-1}\\
			&= v[p]_{j} \cdots v[p]_{j + n - 1}\\
			&= v[p]_{j} \cdots v[p]_{n} v[p]_{1} \cdots v[p]_{j - 1}\\
			&= w[p]_j \cdots w[p]_n w[p]_1 \cdots w[p]_{j-1}\end{align*}
		is just a permutation of the digits of $w[p]_1 \cdots w[p]_n$, and since $w$ is an $(a,n,n+r)$-perfect necklace, $w[p]_1 \cdots w[p]_n$ visits all words in $\A^n$ as $p$ varies from $1$ to $a^n$. Thus $v[p]_i \cdots v[p]_{i+n-1}$ visits all words in $\A^n$ as $p$ varies from $1$ to $a^n$ as well.
		
		Next, consider $i \in \set{nq-n+2,\ldots,nq+r-n+1}$. Let $j=i-n(q-1)\in\set{2,\ldots,r+1}$. Then
		\begin{align*}v[p]_{i} \cdots v[p]_{i+n-1} &= v[p]_{n(q-1)+j} \cdots v[p]_{n(q-1)+j+n-1}\\
			&= w[p]_{j} \cdots w[p]_{n+j-1}.\end{align*}
		Again, since $w$ is an $(a,n,n+r)$-perfect necklace, this visits all words of $\A^n$ as $p$ varies.
		
		Finally, consider $i \in \set{nq+r-n+2, \ldots, nq+r}$. Once again, let $j=i-n(q-1)\in\set{r+2,\ldots,n+r}$. Then
		\begin{align*}v[p]_{i} \cdots v[p]_{nq + r} v[p+1]_1 \cdots v[p+1]_{i - n(q-1) - r - 1} &= v[p]_{n(q-1)+j} \cdots v[p]_{nq + r} v[p+1]_1 \cdots v[p+1]_{j - r - 1}\\
			&= w[p]_j \cdots w[p]_{n+r}w[p+1]_1 \cdots w[p+1]_{j - r -1}.\end{align*}
		Once again, since $w$ is an $(a,n,n+r)$-perfect necklace, this visits all words of $\A^n$ as $p$ varies.
	\end{proof}
	
	In the remainder of this section, for convenience in these particular results, we index from 0.
	
	\begin{prop}\label{prop:pn_const1}
		Let $a,n \in \mathbb{N}$ with $a\geq 2$, and let $w = (w_0 \dots w_{a^n-1})$ be a De Bruijn cycle for $\set{0,1, \ldots, a-1}^n$. Let $r \in \set{1,\ldots,n}$ satisfy  $\mathrm{gcd}(a^n, r) = 1$. For $p \in \set{0, 1, \ldots, a^n - 1}$, let $w[p] = w_{rp} \cdots w_{rp + n-1} w_{rp} \cdots w_{rp + r-1}$, (with indices taken modulo $a^n$),
		then $v=(w[0]\cdots w[a^n-1])$ is an $(a, n, n+r)$-perfect necklace.
	\end{prop}

	\begin{proof}
		Let $j\in\{0,\dots,n+r-1\}$. We show that as $p$ ranges from $0$ to $a^{n}-1$, $v_{(n+r)p + j}\ldots v_{(n+r)p + j + n-1}$ ranges over every word in $\{0,\dots,a-1\}^n$, which proves that $v$ is a perfect necklace.

		First we consider the case where $j \in \set{0, \ldots , r-1}$. Note that the length of each word $w[p]$ is $n+r$, so we have that for $p\in \set{0, 1, \ldots, a^n - 1}$, $v_{(n+r)p} = w[p]_0=w_{rp}$. We then have that
		\begin{align*}
			v_{(n+r)p + j}\ldots v_{(n+r)p + j + n-1} &= w_{rp+ j} \ldots w_{rp + n-1} w_{rp} \ldots w_{rp + j - 1}
		\end{align*}
		(which is interpreted to be $w_{rp} \ldots w_{rp + n-1}$ when $j = 0$) is just a cyclic permutation of $w_{rp} \ldots w_{rp + n-1}$. Since $\mathrm{gcd}(a^n, r) = 1$ and $w$ is a De Bruijn cycle, as $p$ ranges from $0$ to $a^n-1$, $w_{rp} \ldots w_{rp + n-1}$ ranges over every window of $w$ and thus every word in $\{0,\dots,a-1\}^n$. And since $w_{rp+ j} \ldots w_{rp + n-1} w_{rp} \ldots w_{rp + j - 1}$ is a cyclic permutation of $w_{rp} \ldots w_{rp + n-1}$, the same is true for $w_{rp+ j} \ldots w_{rp + n-1} w_{rp} \ldots w_{rp + j - 1}$ as $p$ ranges from $0$ to $a^n-1$, completing the proof in this case.
		
		If $j \in \set{r, \ldots , n-1}$, then
		\begin{align*}
			v_{(n+r)p + j}\ldots v_{(n+r)p + j + n-1} &= w_{rp+ j} \ldots w_{rp + n-1} w_{rp} \ldots w_{rp + r - 1}w_{r(p+1)} \ldots w_{r(p+1) + j-r-1}\\
			&=w_{rp+ j} \ldots w_{rp + n-1} w_{rp} \ldots w_{rp + r - 1}w_{rp+r} \ldots w_{rp + j-1}\\
			&=w_{rp+ j} \ldots w_{rp + n-1} w_{rp} \ldots w_{rp + j-1}
		\end{align*}
		(where $w_{r(p+1)} \ldots w_{r(p+1) + j-r-1}$ is similarly interpreted to be the empty word when $j=r$) is again a cyclic permutation of $w_{rp} \ldots w_{rp + n-1}$. Therefore, as above, as $p$ ranges from $0$ to $a^{n}-1$, $w_{rp+ j} \ldots w_{rp + n-1} w_{rp} \ldots w_{rp + j-1}$ ranges over every word in $\{0,\dots,a-1\}^n$.
		
		Now we consider the final case, where $j \in \set{n, \ldots, n+r - 1}$. Set $i = j-n$. Then
		\begin{align*}
			v_{(n+r)p + j}\ldots v_{(n+r)p + j + n-1} &= w_{rp+ i} \ldots w_{rp + r-1} w_{r(p+1)} \ldots w_{r(p+1) + i - r + n -1} \\
			&= w_{rp+i} \ldots w_{rp+ i + n -1}
		\end{align*}
		which is just a length-$n$ subword of $w$. Therefore, again, as $p$ ranges from $0$ to $a^{n}-1$, $w_{rp+i} \ldots w_{rp+ i + n -1}$ ranges over every word in $\{0,\dots,a-1\}^n$.
	\end{proof}
	
	\begin{example}
		Take the De Bruijn cycle $w=(00011101)$ for $\{0,1\}^3= \set{0,\ldots,a-1}^n$. 
		Let $r=1$. The construction in Proposition~\ref{prop:pn_const1} gives
		$$
		v = (0000 ~ 0010 ~ 0110 ~ 1111 ~ 1101 ~ 1011 ~ 0100 ~ 1001)
		$$
		which is a $(2,3,4)$-perfect necklace. (Additional spacing is provided above for the reader's convenience.)
	\end{example}
	
	Proposition \ref{prop:pn_const2} below provides another method to construct a perfect necklace when a De Bruijn cycle is given. Specifically, when a De Bruijn cycle for $\{0,\dots,a-1\}^n$ is given, Proposition \ref{prop:pn_const2} produces an $(a,n,2n-1)$-perfect necklace.
	
	\begin{prop}\label{prop:pn_const2}
		Let $a,n \in \mathbb{N}$ with $a\geq 2$, and let $w = (w_0 ,\dots, w_{a^n-1})$ be a De Bruijn cycle for $\set{0,1, \ldots, a-1}^n$. For $p \in \set{0, 1, \ldots, a^n - 1}$, let $w[p] = w_{-p} \ldots w_{-p + n-1} w_{-p} \ldots w_{-p + n -2}$, (with indices taken modulo $a^n$) for $p \in \set{0, 1, \ldots, a^n - 1}$,
		then $v=(w[0]\cdots w[a^n-1])$ is an $(a, n, 2n-1)$-perfect necklace.
	\end{prop}
	
	\begin{proof}
		This is proved in a very similar way as the proof of Proposition \ref{prop:pn_const1}.
		
		We first consider the case that $j \in \set{0, \ldots , n-1}$, then for $p \in \set{0, 1, \ldots, a^n - 1}$,
		\begin{align*}
			v_{(2n-1)p + j}\ldots v_{(2n-1)p + j + n-1} &= w_{-p+ j} \ldots w_{-p + n-1} w_{-p} \ldots w_{-p + j - 1}
		\end{align*}
		(when $j = 0$, this is interpreted to just be $w_{-p} \ldots w_{-p + n-1}$) is just a cyclic permutation of $w_{-p} \ldots w_{-p + n-1}$, an $n$-window of $w$. Therefore, by similar reasoning as in the proof of Proposition \ref{prop:pn_const1} as $p$ ranges from 0 to $a^n-1$, $w_{-p+ j} \ldots w_{-p + n-1} w_{-p} \ldots w_{-p + j - 1}$ ranges over all words of $\{0,\dots,a-1\}^n$, completing the proof in this case.
		
		If $j \in \set{n, \ldots, 2n-2}$, then we first set $i = j-n$. We have that
		\begin{align*}
			v_{(2n-1)p + j}\ldots v_{(2n-1)p + j + n-1} &= w_{-p+ i} \ldots w_{-p + n - 2} w_{-(p+1)} \ldots w_{-(p+1) + i} \\
			&=w_{-p+ i} \ldots w_{-p + n - 2} w_{-p-1} \ldots w_{-p + i - 1}
		\end{align*}
		is just a cyclic permutation of $w_{-p-1} \ldots w_{-p + n-2}$. So letting $p$ vary, we get this cyclic permutation for each length $n$ subword of $w$. So again as $p$ ranges from 0 to $a^n-1$, $w_{-p+ i} \ldots w_{-p + n - 2} w_{-p-1} \ldots w_{-p + i - 1}$ ranges over all words of $\{0,\dots,a-1\}^n$, completing the proof.
	\end{proof}
	
	\begin{example}
		The cycle $w=(0011)$ is a De Bruijn cycle for $\{0,1\}^2 = \{0,\dots, a-1\}^n$. Applying the construction from Proposition~\ref{prop:pn_const2} gives
		$$
		v = (000 ~ 101 ~ 111 ~ 010)
		$$
		which is a $(2,2,3)$-perfect necklace. (Additional spacing is provided above for the reader's convenience.)
	\end{example}
	
	\section{Alphabet Multiplier}\label{sec:alph_mult}
	
	We begin by defining a few operations on (cyclic) partial words. For $w$ a (cyclic) partial word over $\set{0,1,2, \ldots, k-1}$ and $a$ an integer, we define $a \cdot w$ to be the (cyclic) partial word over $\set{0,a,2a, \ldots, (k-1)a}$ given by multiplying each character of $w$ by $a$, where $a \cdot \diam := \diam$. If $w$ is a (cyclic) partial word over an alphabet of nonnegative integers, then we take $w$ mod $a$ to be the (cyclic) partial word obtained by replacing each non-diamond character by its remainder modulo $a$. If $v$ and $w$ are (cyclic) partial words of equal length with the same frame,
	then we define $v + w$ to be the (cyclic) partial word given by adding the characters of $v$ and $w$ componentwise, where $\diam + \diam := \diam$. If $w$ is a (cyclic) partial word, then we denote the word obtained by removing the diamond characters by $w \setminus \diam$. Lastly, if $w$ is a (cyclic) partial word, recall that $w^q$ is the concatenation of $q$ copies of $w$. 
	
	The following theorem gives a construction of an upcycle for $\set{0,1,\ldots,ak-1}^n$ with diamondicity $d$ provided an upcycle for $\set{0,1,\ldots,a-1}^n$ with diamondicity $d$ is given.
	
	\begin{theorem}[Alphabet Multiplier]\label{thm: alphamult}
		Let $u$ be an upcycle for $\set{0,1,\ldots,a-1}^n$ with diamondicity $d$, and let $k \ge 1$. Let $v$ be a cyclic partial word over $\set{0,1,\ldots,k-1}$ with the same length and frame as $u^{(k^{n-d})}$. Then $a \cdot v + u^{(k^{n-d})}$ is an upcycle for $\set{0,1,\ldots,ak-1}^n$ if and only if $v\setminus \diam$ is a $(k, n-d, \frac{(n-d)a^{n-d}}{n})$-perfect necklace.
	\end{theorem}
	\begin{proof}
		Note that $\frac{(n-d)a^{n-d}}{n}$ is the number of non-diamond characters in $u$, so $u^{(k^{n-d})}$ (and thus $v$) has $\frac{(n-d)a^{n-d}}{n}k^{n-d}$ non-diamond characters. Therefore $v\setminus \diam$ has the length of a $(k,n-d, \frac{(n-d)a^{n-d}}{n})$-perfect necklace.
		
		Let $w$ be the cyclic partial word $a \cdot v + u^{(k^{n-d})}$. Note $w$ has length $(ak)^{n-d}$, alphabet $\set{0,1,\ldots,ak-1}$, and $d$ diamonds in each $n$-window, i.e., $w$ has the length and the frame of an upcycle for $\set{0,1,\ldots,ak-1}^n$ with diamondicity $d$. Therefore $w$ is an upcycle for $\set{0,1,\ldots, ak-1}^n$ if and only if $w$ covers each word of $\set{0,1,\ldots, ak-1}^n$ at most once.
		
		Since $w \text{ mod } a$ is $u^{(k^{n-d})}$, if a word of length $n$ is covered more than once by $w$, then it is covered by windows at positions which are some multiple of $\abs{u} = a^{n-d}$ apart. These windows share the same frame, and therefore must be the same substring of $w$ if they are to cover the same word. So $w$ covers a word more than once if and only if $w$ (and therefore $v$) contains a repeated substring at positions which are a multiple of $a^{n-d}$ apart. This occurs exactly when $v \setminus \diam$ contains a repeated length $n-d$ word at positions which are a multiple of $\frac{(n-d)a^{n-d}}{n}$ apart. Hence $w$ is an upcycle if and only if $v\setminus \diam$ is a $(k, n-d, \frac{(n-d)a^{n-d}}{n})$-perfect necklace.
	\end{proof}
	
	\begin{corollary}\label{cor:alphamult}
		Let $u$ be an upcycle for $\set{0,1,\ldots,a-1}^n$ with diamondicity $d$. Then for every $k \geq 1$, we can construct from $u$ an upcycle for $\set{0,1,\ldots,ak-1}^n$ with diamondicity $d$.
	\end{corollary}
	\begin{proof}
		By Proposition~\ref{prop:PerfNecksExist!}, there exists a perfect necklace for any parameters. Further, any $(k, n-d, \frac{(n-d)a^{n-d}}{n})$-perfect necklace can be padded with diamonds to a partial cycle $v$ of the same length and frame as $u^{(k^{n-d})}$. Thus by Theorem \ref{thm: alphamult}, $a \cdot v + u^{(k^{n-d})}$ is an upcycle for $\set{0,1,\ldots,ak-1}^n$ with diamondicity $d$.
	\end{proof}
	
	\begin{corollary}\label{cor:n=8&a_even}
		For every even integer $a \ge 2$, there is an upcycle for $\set{0,\ldots,a-1}^8$.
	\end{corollary}
	
	\begin{proof}
		By Theorem \ref{n=8d=1} and Corollary \ref{cor:alphamult}.
	\end{proof}
	
	\begin{example}\label{ex: AlphMult}
		Consider the upcycle $u= (001\diam110\diam)$, with parameters $a=2$, $n=4$, and $d=1$. Take $k=2$.
		
		Then $n-d=3$ and $\frac{(n-d)a^{n-d}}{n} = 6$, so we need a partial cycle $v$ with the same frame as $u$ such that $v \setminus \diam$ is a $(2,3,6)$-perfect necklace. One such $v$ is
		\[
		v = (000 \diam 000 \diam 
		001 \diam 001 \diam 
		010 \diam 010 \diam 
		011 \diam 011 \diam 
		100 \diam 100 \diam 
		101 \diam 101 \diam 
		110 \diam 110 \diam 
		111 \diam 111 \diam ).
		\]
		Applying Theorem~\ref{thm: alphamult}, we get the following:
		\begin{align*}
			a \cdot v               &= (000 \diam 000 \diam 
			002 \diam 002 \diam 
			020 \diam 020 \diam 
			022 \diam 022 \diam 
			200 \diam 200 \diam 
			202 \diam 202 \diam 
			220 \diam 220 \diam 
			222 \diam 222 \diam) \\
			u^{k^{n-d}}             &= (001 \diam 110 \diam
			001 \diam 110 \diam
			001 \diam 110 \diam
			001 \diam 110 \diam
			001 \diam 110 \diam
			001 \diam 110 \diam
			001 \diam 110 \diam
			001 \diam 110 \diam) \\
			a \cdot v + u^{(k^{n-d})} &= (001 \diam 110 \diam
			003 \diam 112 \diam
			021 \diam 130 \diam
			023 \diam 132 \diam
			201 \diam 310 \diam
			203 \diam 312 \diam
			221 \diam 330 \diam
			223 \diam 332 \diam).
		\end{align*}
		Thus $w=a \cdot v + u^{(k^{n-d})}$ is an upcycle for $\set{0,1,2,3}^4$ with $d=1$.
	\end{example}
	
	Observe that the upcycle $w$ constructed in Example~\ref{ex: AlphMult} contains a substring $s$ for which $u=(s)$. This is due to the fact that the perfect necklace $v \setminus \diam$ contains a string of 6 consecutive zeros. Below we show that such perfect necklaces exist in general; thus we can use Theorem~\ref{thm: alphamult} to create new upcycles for larger alphabets that contain the original upcycle. 
	
	\begin{prop}\label{thm:StartingWith0sExists}
		For all $a,n,t \in \N$ there exists an $(a,n,t)$-perfect necklace which begins with $t$ zeroes.
	\end{prop}
	\begin{proof}
		An Euler tour of the astute graph $G(a,n-1,t)$ describes an $(a,n,t)$-perfect necklace by reading the labels of the edges in order. Recall that the astute graph $G(a,n-1,t)$ is a tensor product of the De Bruijn graph $B(a,n-1)$ with the $t$-cycle $C_t$. Construct an Euler tour of $G(a,n-1,t)$ by choosing, as the initial vertex, any of the $t$ vertices corresponding to $0^{n-1}$. Proceed once around the cycle, through every vertex corresponding to $0^{n-1}$ and back to the initial vertex.
		
		If $a=1$, then this cycle is an Euler tour, hence corresponds to a perfect necklace. Otherwise, remove the edges of this cycle. The remaining graph has $\text{in-degree} = \text{out-degree}$ for each vertex and is weakly connected, so has an Euler tour. Joining the cycle and an Euler tour of the rest of the graph yields an Euler tour of $G(a,n-1,t)$ which corresponds to a perfect necklace beginning with $t$ zeroes.
	\end{proof}
	
	\begin{corollary}\label{cor:onion}
		For an upcycle $u$ with parameters $(a,n,d)$ and an integer $k\geq 1$, there exists an upcycle with parameters $(ka,n,d)$ which contains the original upcycle.
	\end{corollary}
	
	\begin{proof}
		Let $k \geq 1$. In Theorem \ref{thm: alphamult}, take $v \setminus \diam$ to be a $(k, n-d, \frac{(n-d)a^{n-d}}{n})$-perfect necklace beginning with $\frac{(n-d)a^{n-d}}{n}$ zeroes, which exists by Proposition \ref{thm:StartingWith0sExists}. The resulting upcycle contains $u$ as an initial substring.
	\end{proof}
	
	\begin{remark} Given an upcycle $u$ with parameters $(a,n,d)$ and an infinite sequence of positive integers $(k_0=1,k_1,\dots)$, repeated applications of Corollary \ref{cor:onion} allow us to construct a partial word of infinite length such that the initial substring of length $(a\prod_{i=0}^jk_i)^{n-d}$ is an upcycle for $\set{0,\ldots, a\prod_{i=0}^jk_i-1}^n$ with diamondicity $d$ for all $j\geq 0$, reminiscent of the onion De Bruijn cycles in \cite{SSW19}. This is in contrast to \cite{BC21} in which a De Bruijn cycle for alphabet size $a+1$ contains a De Bruijn cycle for alphabet size $a$ as a subsequence, but not consecutively, so that the new symbol appears frequently.
	\end{remark}

	Returning to Example~\ref{ex: AlphMult}, we also note the overall structure of $v \setminus \diam.$ In particular, observe that $v \setminus \diam$ is the concatenation of the words $w^2$ for $w \in \{0,1\}^3$ in lexicographic order. We will now show that for any upcycle, a similar perfect necklace can be used to construct $v$. In addition we show that this construction has a nice description when the original upcycle $u$ has parameters $(a,n,d)$ where $n \mid a^{n-d}.$

	\begin{theorem}\label{thm:divalphmult}
		Let $k \geq 1$. Suppose $u$ is a universal partial cycle for $\set{0,1,\ldots, a-1}^n$ with diamondicity $d$ such that $n \mid a^{n-d}$. Then $a \cdot v + u^{k^{n-d}}$ is a universal partial cycle for $\set{0, 1, \ldots, ak-1}$, where $v \setminus \diam$ is the concatenation of $w^\frac{a^{n-d}}{n}$ for the words $w \in \set{0,1,\ldots, k-1}^{n-d}$ in lexicographic order.
	\end{theorem}
	
	\begin{proof}
		By Lemma \ref{lemma:LexOrderPerfNeck}, if we take $w[j]$ to be the $j$th word of $\set{0,1,\ldots, k-1}^{n-d}$ according to the lexicographic order, then $(w[1]\cdots w[k^{n-d}])$ is a $(k,n-d,n-d)$-perfect necklace. By Proposition \ref{prop: smallperfect}, the cyclic word $(w[1]^\frac{a^{n-d}}{n}\ldots w[k^{n-d}]^\frac{a^{n-d}}{n})$ is a $(k,n-d,\frac{(n-d)a^{n-d}}{n})$-perfect necklace. Now by taking $v\setminus \diam$ to be this $(k,n-d,\frac{(n-d)a^{n-d}}{n})$-perfect necklace, we may apply Theorem \ref{thm: alphamult} to obtain the result.
	\end{proof}
	
	We observe that $v$ in Example~\ref{ex: AlphMult} follows the construction described in Theorem~\ref{thm:divalphmult}. We also note the following.
	
	\begin{observation}
		Applying Theorem~\ref{thm:divalphmult} to the upcycle $(001\diam110\diam)$ for $\set{0,1}^4$ yields the construction of universal partial words for even size alphabets and $n=4$ given in \cite[Theorem~5.2]{G18}.
	\end{observation}
	
	The only known non-trivial upcycles 
	as of writing have $d=1$, hence satisfy the condition $n\mid a^{n-d}$. 
	Thus, Theorem~\ref{thm:divalphmult} can be applied to any currently known non-trivial upcycle. However, we can apply Theorem~\ref{thm: alphamult} to all upcycles regardless of whether $n \mid a^{n-d}$ (in particular, to all De Bruijn cycles). We give an example application to a De Bruijn cycle with $n \nmid a^n$ below.

	\begin{example} \label{ex: fixed}
		Consider the De Bruijn cycle $u=(11101000)$ with parameters $a=2,$ $n=3$, and $d=0$. Take $k=2$.
		
		Thus $n-d = 3$ and $\frac{(n-d)a^{n-d}}{n} = 8,$ so we need a partial cycle $v$ with the same frame as $u$ such that $v \setminus \diam = v$ is a $(2,3,8)$-perfect necklace. An example of such a perfect necklace $v$ is given by 
		
		\[
		v= (11111111 ~ 01101101 ~ 00100100 ~ 00000000 ~ 10010010 ~ 01001001 ~ 10110110 ~ 11011011).
		\]
		
		(The spaces above are provided for ease of reading.)
		
		Applying Theorem~\ref{thm: alphamult}, we get the following:
		\begin{align*}
			a \cdot v               &= (22222222 ~ 02202202 ~ 00200200 ~ 00000000 ~ 20020020 ~ 02002002 ~ 20220220 ~ 22022022) \\
			u^{k^{n-d}}             &= (11101000 ~ 11101000 ~ 11101000 ~ 11101000 ~ 11101000 ~ 11101000 ~ 11101000 ~ 11101000) \\
			a \cdot v + u^{k^{n-d}} &= (33323222 ~ 13303202 ~ 11301200 ~ 11101000 ~ 31121020 ~ 13103002 ~ 31321220 ~ 33123022).
		\end{align*}
		
		Then $w=a \cdot v + u^{k^{n-d}}$ is a De Bruijn cycle with parameters $a=4$, $n=3$, and $d=0$. 
	\end{example}

	\section{Lifts and Folds of Upcycles}\label{sec:lift}
	\subsection{How to lift an upcycle}\label{subsec:lift&fold}

	Given an upcycle with diamondicity $d$, in this section we consider closely related upcycles with diamondicity $d' < d$. We focus on the following definition.
	
	\begin{definition}\label{def:fold/lift}
		Let $u$ and $w$ be upcycles for $\A^n$ having diamondicities $d$ and $d'$, respectively, where $d' < d$. We say that $w$ is a \emph{lift} of $u$ (or that $u$ \emph{lifts} to $w$) if $w$ is covered by $u^{(a^\delta)}$, where $\delta := d - d'$. In this situation, we also say that $u$ is a \emph{fold} of $w$ (or that $w$ \emph{folds} to $u$).
		
		If $w$ is a De Bruijn sequence, then we say that $w$ is a \emph{De Bruijn lift} of $u$.
	\end{definition}
	
	Lifts are transitive: if $w$ is a lift of $v$ and $v$ is a lift of $u$, then $w$ is a lift of $u$.
	
	\begin{example}\label{ex:first_folds}
		Let $\A = \{0,1\}$ and $n=4$. Then $u=(001\diam 110 \diam)$ is an upcycle for $\A^n$ with $d=1$. Both $w= (0010110000111101)$ and $x=(0010111101001100)$ are upcycles with $d=0$ (i.e., De Bruijn cycles) for $\A^n$. It is straightforward to verify that $w$ is a lift of $u$ while $x$ is not a lift of $u$.
	\end{example}
	
	If $u$ and $w$ are upcycles with diamondicities $d$ and $d'$ respectively, and $w$ is a lift of $u$, then by definition $w$ is $u^{(a^{d-d'})}$ with some $n$-periodic set of diamonds in $u$ replaced by a cyclic word $v$. This suggests a method for producing lifts of $u$. 
	In Theorem \ref{thm:necklacelift} we state this method and show that all lifts of upcycles are generated using perfect necklaces in this way (see also Remark \ref{rem:liftchar}).
	
	\begin{theorem}[Characterization of lifts]\label{thm:necklacelift}
		Let $u$ be an upcycle for $\A^n$ with diamondicity $d$. Select an $n$-periodic subset of the diamonds in $u$, and let $\delta$ be the number of selected diamonds per window. Let $w$ be the cyclic partial word obtained by taking the cyclic partial word $u^{(a^{\delta})}$ and replacing all the selected diamonds (and their copies) with a cyclic word $v$. Then $w$ is an upcycle if and only if $v$ is an $(a, \delta, \frac{\delta a^{n-d}}{n})$-perfect necklace. If so, then $w$ is a lift of $u$ with diamondicity $d - \delta$.
	\end{theorem}

	\begin{proof}
		Note that $\frac{\delta a^{n-d}}{n}$ is the number of selected diamonds in $u$, so $v$ has length $\frac{\delta a^{n-d}}{n} a^{\delta}$, the length of an $(a,\delta, \frac{\delta a^{n-d}}{n})$-perfect necklace.
		
		Let $w$ be as in the theorem. Since $w$ has length $a^{n-d+\delta}$ with $d-\delta$ diamonds per window, it covers $a^{n-d+\delta} \cdot a^{d-\delta} = a^n$ words of length $n$ with multiplicity. It is sufficient to show that $w$ covers each word of $\A^n$ at most once if and only if $v$ is a $(a,\delta, \frac{\delta a^{n-d}}{n})$-perfect necklace.
		
		Since $w$ is covered by $u^{(a^\delta)}$, and $u$ is an upcycle, if a word of $\A^n$ is covered more than once by $w$, then it is covered by windows at positions which are a multiple of $a^{n-d}$ apart. These windows correspond to two copies of the same window of $u$ with the same diamonds removed, so share the same frame, and therefore must be identical substrings of $w$ if they are to cover the same word. The corresponding substrings of $w$ of length $n$ can only differ by words of length $\delta$ that appear in $v$ at positions which are a multiple of $\frac{\delta a^{n-d}}{n}$ apart. So $w$ covers a word more than once if and only if $v$ contains a repeated word of length $\delta$ at positions which are a multiple of $\frac{\delta a^{n-d}}{n}$ apart. Hence $w$ is an upcycle if and only if $v$ is an $(a, \delta, \frac{\delta a^{n-d}}{n})$-perfect necklace.
	\end{proof}
	
	\begin{remark}\label{rem:liftchar}
		Here, we argue why Theorem \ref{thm:necklacelift} is indeed a characterization of all lifts; namely, we show that an $n$-periodic set of diamonds chosen from $u^{(a^\delta)}$ corresponds to $a^\delta$ copies of an $n$-periodic set of diamonds from $u$.
		Assume that an $n$-periodic set of diamonds, with $\delta$ diamonds in each window, have been chosen from $u^{(a^\delta)}$, and, to focus on the periodicity of these diamonds, consider the cyclic word $w$ over $\{\bull,\diam\}$ obtained by replacing each diamond in the frame of $u^{(a^\delta)}$ that was not chosen with $\bull$. The frame period (see Definition \ref{def:nonexistence}) $m$ of $w$ divides $\gcd(n,a^{n-d'})$ since $w$ is both $n$-periodic and $a^{n-d'}$-periodic with $d' := d-\delta$ (or by \cite[Theorem 4.11]{G18}). If $\gcd(n,a^{n-d'})>\gcd(n,a^{n-d})$, then in particular there is a prime power $p^m$ dividing $n$ where $m> n-d$. By Theorem 4.7 in \cite{G18}, we have $d< n-\sqrt{n-\frac{7}{4}} - \frac{1}{2}$ and hence $m > \left\lfloor  \sqrt{n-\frac{7}{4}} + \frac{1}{2} +1 \right\rfloor$. However, $\left\lfloor  \sqrt{n-\frac{7}{4}} + \frac{1}{2} +1 \right\rfloor > \log_2(n)$ for all $n\geq 1$, so $p^m > n$, contradicting that $p^m$ divides $n$. Thus we indeed have $\gcd(n,a^{n-d'})=\gcd(n,a^{n-d})$. 
		Let $k=a^{n-d}/m$, which is an integer because $m$ divides $\gcd(n,a^{n-d'})=\gcd(n,a^{n-d})$. Then $w=(w_1\cdots w_m)^{a^{n-d'}/m}=((w_1\cdots w_m)^k)^{a^{\delta}}$. Note $(w_1\cdots w_m)^k$ and $u$ have the same length; the former indicates the chosen diamonds and is a power of a word of length $m$, so is $m$-periodic and therefore $n$-periodic as a cyclic word since $m$ divides $n$. 
		Thus the $n$-periodic set of diamonds chosen from $u^{(a^\delta)}$ is obtained by first choosing an $n$-periodic set of diamonds from $u$ and taking all copies of these chosen diamonds in $u^{(a^\delta)}$. 
		Therefore, by Theorem \ref{thm:necklacelift}, every lift of $u$ with diamondicity $d - \delta$ can be constructed from a $(a, \delta, \frac{\delta a^{n-d}}{n})$-perfect necklace.
	\end{remark}
	
	\begin{corollary}\label{cor:lift_to_DB}
		Every upcycle lifts to a De Bruijn cycle.
	\end{corollary}
	
	\begin{proof}
		Apply Theorem \ref{thm:necklacelift} with $\delta = d$, as $(a, d, \frac{d a^{n-d}}{n})$-perfect necklaces exist by Proposition~\ref{prop:PerfNecksExist!}.
	\end{proof}
	
	\begin{remark}
		Whether there exist upcycles with diamondicity $d > 1$ is still an open problem, but if such an upcycle $u$ exists and if $n\nmid a^{n-d}$, then any $n$-periodic subset of diamonds in $u$ will have $\delta>1$. That is, any lift of $u$ must have diamondicity $d'$ with $d'<d-1$.
		
		In the case where $d = 1$, producing a lift requires only an $(a,1, \frac{a^{n-1}}{n})$-perfect necklace. When constructing a lift of such an upcycle $u$, the diamonds can be filled in greedily, requiring only that each copy in $u^a$ of a given diamond from $u$ be replaced with a distinct character.
	\end{remark}
	
	\begin{example}\label{ex: lifts}
		The upcycle $u = (001\diam110\diam)$ for $\set{0,1}^4$ has $d=1$. Taking $\delta = 1$, we have
		\begin{align*}
			u^{\parens{a^\delta}} = u^2 = (&001\diam110\diam 001\diam110\diam)
		\end{align*}
		and $\frac{\delta a^{n-d}}{n} = \frac{2^3}{4}=2.$ Therefore we need a $(2,1,2)$-perfect necklace to fill in the selected diamonds in the above word. 
		Table~\ref{fig:lift_exs} shows all $(2,1,2)$-perfect necklaces and their resulting lifts of $u$.
		
		\begin{table}[H]
			\begin{center}
				\begin{tabular}{c|c}
					Perfect Necklace & Lift of $u$ \\
					\hline
					$(00 ~ 11)$      & $(001\,\mathbf{0}\,110\,\mathbf{0}\,001\,\mathbf{1}\,110\,\mathbf{1})$ \\
					$(01 ~ 10)$      & $(001\,\mathbf{0}\,110\,\mathbf{1}\,001\,\mathbf{1}\,110\,\mathbf{0})$ \\
					$(10 ~ 01)$      & $(001\,\mathbf{1}\,110\,\mathbf{0}\,001\,\mathbf{0}\,110\,\mathbf{1})$ \\
					$(11 ~ 00)$      & $(001\,\mathbf{1}\,110\,\mathbf{1}\,001\,\mathbf{0}\,110\,\mathbf{0})$
				\end{tabular}
			\end{center}
			\caption{Each $(2,1,2)$-perfect necklace $v$ corresponds to a lift of $u=(001\diam110\diam)$ by replacing the diamonds in $u^2=(001\diam110\diam001\diam110\diam)$ with the characters of $v$.}
			\label{fig:lift_exs}
		\end{table}
		
		Observe that the first and fourth lifts in Table~\ref{fig:lift_exs} are the same cycle, and the second and third are the same. Thus there are two De Bruijn cycles that fold to $u$. Corresponding graphical representations of these cycles are shown in Figures~\ref{fig: circle1 lift1} and \ref{fig: circle1 lift2} in Section~\ref{subsec:Hamilton} as well as in Figure \ref{fig:tour1twolifts} in Section \ref{subsec:Euler}.
	\end{example}

	\subsection{Upcycles as generalized Hamilton cycles}\label{subsec:Hamilton}
	
	De Bruijn cycles for $\A^n$ can be viewed as Hamilton cycles of the De Bruijn graph $B(a,n)$. The Hamilton cycle is the spanning subgraph of the De Bruijn graph indicating the order in which the De Bruijn cycle covers the words of $\A^n$, with an edge from $x$ to $y$ if and only if $x$ and $y$ are covered consecutively in the De Bruijn cycle.
	
	We now define the analogous graph substructure representing an upcycle.
	
	\begin{definition}
		Given an upcycle $u$ for $\A^n$, let $S(u)$ be the spanning subgraph of the De Bruijn graph for $\A^n$ 
		having an edge from $x$ to $y$ if and only if $x\to y$ is an edge of the De Bruijn graph and the words $x$ and $y$ are covered consecutively in $u$. The edges of $S(u)$ are labeled with the letter that is appended in going from $x$ to $y$.
	\end{definition} 
	
	In Figures \ref{fig: circle1} and \ref{fig: circle2}, we show the graphs $S(u)$ for the upcycles $u=(001\diam110\diam)$ and \[u = (001\diam110\diam003\diam112\diam021\diam130\diam023\diam132\diam201\diam310\diam203\diam312\diam221\diam330\diam223\diam332\diam),\] with edge colors indicating edge labels and the characters of $u$ (equivalent to edge labels for the letters) written around the outside.
	
	\begin{figure}[H]
		\begin{center}
			\begin{tikzpicture}
				
				\foreach \x [count=\i] in {0,0,1,$\diam$,1,1,0,$\diam$}
				{
					{\foreach \y in {1.5,3}
						{\color{\x} \draw [ultra thick] (112.5-\i*45:\y) arc (112.5-\i*45:67.5-\i*45:\y);}
						{\draw [ultra thick] (90-\i*45:3.4) node {\Large{\x}};}
					}
				}
				
				\foreach \x in {1,2}
				{
					\color{0}
					{\draw [ultra thick] (112.5+\x*180:1.5) arc (112.5+\x*180:67.5+\x*180:1.5);}
					{\draw [ultra thick] (112.5+\x*180:3) -- (67.5+\x*180:1.5);}
					
					\color{1}
					{\draw [ultra thick] (112.5+\x*180:1.5) -- (67.5+\x*180:3);};
					{\draw [ultra thick] (112.5+\x*180:3) arc (112.5+\x*180:67.5+\x*180:3);}
				};
				
				\foreach \x in {1,...,8}
				{
					{\node[shape=circle,fill=black, scale=0.5] at (-22.5+\x*45:1.5) {};}
					{\node[shape=circle,fill=black, scale=0.5] at (-22.5+\x*45:3) {};}
				};
				
				\foreach \x [count=\i] in {0001,0010,0101,1011,0110,1100,1000,0000}
				{\draw (-22.5-\i*45:1.91) node {\footnotesize \x};}
				
				\foreach \x [count=\i] in {1001,0011,0111,1111,1110,1101,1010,0100}
				{\draw (-22.5-\i*45:3.41) node {\footnotesize \x};}
				
				\directionarrow
				
			\end{tikzpicture}\\
			\caption{The graph $S(u)$ for $u=(001\diam110\diam).$ Here $a=2$ and $n=4$. All edges are oriented clockwise. 
			}
			\label{fig: circle1}
		\end{center}
	\end{figure}
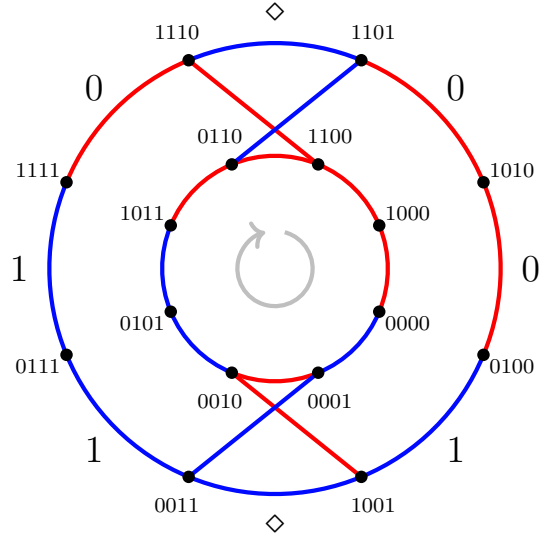
	
	\begin{figure}[H]
		\begin{center}
			\begin{tikzpicture}
				
				\foreach \x [count=\i] in {0,0,1,$\diam$,1,1,0,$\diam$,0,0,3,$\diam$,1,1,2,$\diam$,0,2,1,$\diam$,1,3,0,$\diam$,0,2,3,$\diam$,1,3,2,$\diam$,2,0,1,$\diam$,3,1,0,$\diam$,2,0,3,$\diam$,3,1,2,$\diam$,2,2,1,$\diam$,3,3,0,$\diam$,2,2,3,$\diam$,3,3,2,$\diam$}
				{
					{\foreach \y in {4,5,6,7}
						{\color{\x} \draw [ultra thick] (92.8125-\i*5.625:\y) arc (92.8125-\i*5.625:87.1875-\i*5.625:\y);}
						
						{\draw [ultra thick] (90-\i*5.625:7.4) node {\x};}}
				}
				
				\foreach \x in {1,...,16}
				{
					\color{0}
					{\draw [thick] (92.8125+\x*22.5:4) arc (92.8125+\x*22.5:87.1875+\x*22.5:4);}
					{\draw [thick] (92.8125+\x*22.5:5) -- (87.1875+\x*22.5:4);};
					{\draw [thick] (92.8125+\x*22.5:6) -- (87.1875+\x*22.5:4);};
					{\draw [thick] (92.8125+\x*22.5:7) -- (87.1875+\x*22.5:4);};
					
					\color{1}
					{\draw [thick] (92.8125+\x*22.5:4) -- (87.1875+\x*22.5:5);};
					{\draw [thick] (92.8125+\x*22.5:5) arc (92.8125+\x*22.5:87.1875+\x*22.5:5);}
					{\draw [thick] (92.8125+\x*22.5:6) -- (87.1875+\x*22.5:5);};
					{\draw [thick] (92.8125+\x*22.5:7) -- (87.1875+\x*22.5:5);};
					
					\color{2}
					{\draw [thick] (92.8125+\x*22.5:4) -- (87.1875+\x*22.5:6);};
					{\draw [thick] (92.8125+\x*22.5:5) -- (87.1875+\x*22.5:6);};
					{\draw [thick] (92.8125+\x*22.5:6) arc (92.8125+\x*22.5:87.1875+\x*22.5:6);}
					{\draw [thick] (92.8125+\x*22.5:7) -- (87.1875+\x*22.5:6);};
					
					\color{3}
					{\draw [thick] (92.8125+\x*22.5:4) -- (87.1875+\x*22.5:7);};
					{\draw [thick] (92.8125+\x*22.5:5) -- (87.1875+\x*22.5:7);};
					{\draw [thick] (92.8125+\x*22.5:6) -- (87.1875+\x*22.5:7);};
					{\draw [thick] (92.8125+\x*22.5:7) arc (92.8125+\x*22.5:87.1875+\x*22.5:7);}
				};
				
				\foreach \x in {1,...,64}
				{
					\foreach \y in {4,5,6,7}
					{\node[shape=circle,fill=black, scale=0.4] at (2.8125+\x*5.625:\y) {};}
				}
				
				\directionarrow
				
			\end{tikzpicture}\\
			\caption{The graph $S(u)$ for $u=(001\diam110\diam003\diam112\diam021\diam130\diam023\diam132\diam201\diam310\diam203\diam312\diam221\diam330\diam223\diam332\diam).$ Here $a=4$ and $n=4$. All edges are oriented clockwise.
			}
			\label{fig: circle2}
		\end{center}
	\end{figure}
	
	Each upcycle $u$ for $\A^n$ defines an equivalence relation on $\A^n$ in which two words are equivalent if and only if they are covered by the same window of $u$. In Figures \ref{fig: circle1} and \ref{fig: circle2}, the equivalence classes correspond to vertices lying along radial spokes of the concentric circles. The edges lying along radial spokes of the concentric circles then also have the same labels, except where they correspond to diamonds in the upcycle.
	
	Let $u$ be any upcycle. For each word $w \in \A^n$, let $\overline{w}$ be the equivalence class containing $w$. Note $\overline{w}$ can be interpreted as the window of $u$ covering $w$. Since a diamond can represent any letter of $\A$, in the digraph $S(u)$, a word $w \in \A^n$ has 
	\[\deg^+(w) = \begin{cases} a & \text{if }\overline{w} \text{ starts with } \diam \\
		1 & \text{if } \overline{w} \text{ starts with a letter}\end{cases}\quad\text{and}\quad \deg^-(w) = \begin{cases} a & \text{if }\overline{w} \text{ ends with } \diam \\
		1 & \text{if } \overline{w} \text{ ends with a letter.}\end{cases}
	\]
	
	The following proposition presents the relationship between an upcycle $u$ and a lift of $u$ in the graph setting.
	
	\begin{prop}\label{prop:liftequivalence}
		Let $u$ and $w$ be upcycles for $\A^n$. Then $w$ is a lift of $u$ if and only if $S(w)$ is a (labeled) subgraph of $S(u)$. 
	\end{prop}
	
	\begin{proof}
		Suppose $S(w)$ is a subgraph of $S(u)$. Note that two total words are covered by the same window of an upcycle $u$ if and only if there exists an $n$-length path in $S(u)$ from some total word $v$ to each of them. Consequently, if two total words are covered by the same window in $w$, then since $S(w)$ is a subgraph of $S(u)$, they must be covered by the same window in $u$. Thus, the equivalence relation defined by $w$ is a refinement of the equivalence relation defined by $u$. By associating the equivalence classes with windows, we see that every $n$-window of $w$ is covered by a unique $n$-window of $u$.
		
		Let $w_1$ and $w_2$ be consecutive windows of $w$. Then there exist total words $v_1$ and $v_2$ covered by $w_1$ and $w_2$ such that $(v_1,v_2)$ is an edge in $S(w)$. Let $u_1$ and $u_2$ be the unique windows of $u$ which cover $w_1$ and $w_2$, respectively. Since $S(w)$ is a subgraph of $S(u)$, $(v_1,v_2)$ is also an edge in $S(u)$, so $u_1$ and $u_2$ must be consecutive windows of $u$. Thus, every pair of consecutive windows of $w$ are covered by consecutive windows of $u$, so $w$ is covered by $u^{(a^\delta)}$ (both of which have length $a^{n-d(w)}$).
		
		Conversely, suppose $u^{(a^\delta)}$ covers $w$. If $v_1, v_2 \in \A^n$ are covered by consecutive windows of $w$, then they are covered by consecutive windows of $u^{(a^\delta)}$ (and therefore of $u$). So $S(w)$ is a subgraph of $S(u)$.
	\end{proof}
	
	\begin{corollary}
		Let $u$ be an upcycle. Then $S(u)$ is Hamiltonian.
	\end{corollary}
	
	\begin{proof}
		By Corollary \ref{cor:lift_to_DB}, $u$ lifts to a De Bruijn cycle $w$. The graph $S(w)$ is a Hamilton cycle of $B(a,n)$ and by Proposition \ref{prop:liftequivalence} is a subgraph of $S(u)$. Thus $S(u)$ is Hamiltonian.
	\end{proof}

	\begin{example}
		Consider again Example \ref{ex: lifts}. The upcycle $u = (001\diam110\diam)$ lifts to two distinct De Bruijn cycles $w_1 = (0010110000111101)$ and $w_2 = (0010110100111100)$. By Proposition \ref{prop:liftequivalence}, $S(w_1)$ and $S(w_2)$ give Hamilton cycles of $S(u)$, which is illustrated below in Figures~\ref{fig: circle1 lift1} and \ref{fig: circle1 lift2}.
		\begin{figure}[H]
			\begin{center}
				\begin{tikzpicture}
					
					\foreach \x [count=\i] in {0,0,1,$\diam$,1,1,0,$\diam$}
					{
						{\foreach \y in {1.5,3}
							{\color{\x} \draw [ultra thick] (112.5-\i*45:\y) arc (112.5-\i*45:67.5-\i*45:\y);}
							{\draw [ultra thick] (90-\i*45:3.4) node {\Large{\x}};}
						}
					}
					
					{
						\color{0}
						{\draw [dashed] (112.5+180:1.5) arc (112.5+180:67.5+180:1.5);}
						{\draw [ultra thick] (112.5+180:3) -- (67.5+180:1.5);}
						{\draw [ultra thick] (112.5+2*180:1.5) arc (112.5+2*180:67.5+2*180:1.5);}
						{\draw [dashed] (112.5+2*180:3) -- (67.5+2*180:1.5);}
						
						\color{1}
						{\draw [ultra thick] (112.5+180:1.5) -- (67.5+180:3);};
						{\draw [dashed] (112.5+180:3) arc (112.5+180:67.5+180:3);}
						{\draw [dashed] (112.5+2*180:1.5) -- (67.5+2*180:3);};
						{\draw [ultra thick] (112.5+2*180:3) arc (112.5+2*180:67.5+2*180:3);}
					}
					\foreach \x in {1,...,8}
					{
						{\node[shape=circle,fill=black, scale=0.5] at (-22.5+\x*45:1.5) {};}
						{\node[shape=circle,fill=black, scale=0.5] at (-22.5+\x*45:3) {};}
					};
					
					\foreach \x [count=\i] in {0001,0010,0101,1011,0110,1100,1000,0000}
					{\draw (-22.5-\i*45:1.91) node {\footnotesize \x};}
					
					\foreach \x [count=\i] in {1001,0011,0111,1111,1110,1101,1010,0100}
					{\draw (-22.5-\i*45:3.41) node {\footnotesize \x};}
					
					\directionarrow
					
				\end{tikzpicture}\\
				\caption{
					The De Bruijn cycle $w_1 = (0010110000111101)$ as a lift of $u = (001\diam110\diam)$; $S(w_1)$ is shown as a subgraph of $S(u)$. The dotted edges are the edges of $S(u)$ not contained $S(w_1)$.}
				\label{fig: circle1 lift1}
			\end{center}
		\end{figure}
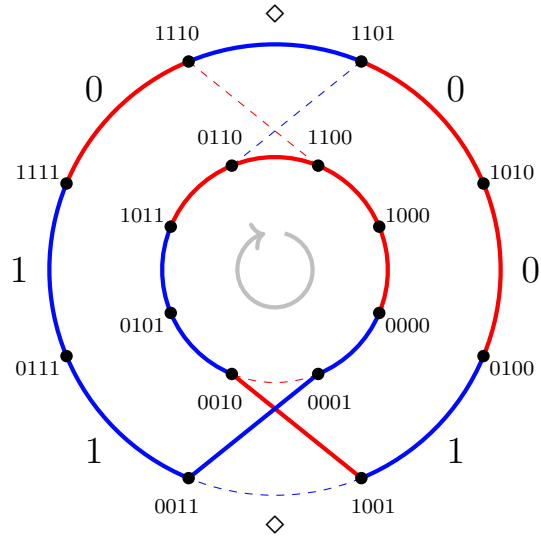 
		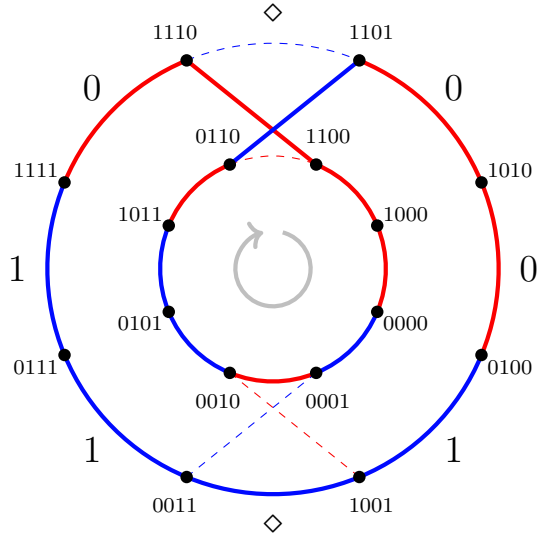
\begin{figure}[H]
			\begin{center}
				\begin{tikzpicture}
					
					\foreach \x [count=\i] in {0,0,1,$\diam$,1,1,0,$\diam$}
					{
						{\foreach \y in {1.5,3}
							{\color{\x} \draw [ultra thick] (112.5-\i*45:\y) arc (112.5-\i*45:67.5-\i*45:\y);}
							{\draw [ultra thick] (90-\i*45:3.4) node {\Large{\x}};}
						}
					}
					
					{
						\color{0}
						{\draw [ultra thick] (112.5+180:1.5) arc (112.5+180:67.5+180:1.5);}
						{\draw [dashed] (112.5+180:3) -- (67.5+180:1.5);}
						{\draw [dashed] (112.5+2*180:1.5) arc (112.5+2*180:67.5+2*180:1.5);}
						{\draw [ultra thick] (112.5+2*180:3) -- (67.5+2*180:1.5);}
						
						\color{1}
						{\draw [dashed] (112.5+180:1.5) -- (67.5+180:3);};
						{\draw [ultra thick] (112.5+180:3) arc (112.5+180:67.5+180:3);}
						{\draw [ultra thick] (112.5+2*180:1.5) -- (67.5+2*180:3);};
						{\draw [dashed] (112.5+2*180:3) arc (112.5+2*180:67.5+2*180:3);}
					}
					\foreach \x in {1,...,8}
					{
						{\node[shape=circle,fill=black, scale=0.5] at (-22.5+\x*45:1.5) {};}
						{\node[shape=circle,fill=black, scale=0.5] at (-22.5+\x*45:3) {};}
					};
					
					\foreach \x [count=\i] in {0001,0010,0101,1011,0110,1100,1000,0000}
					{\draw (-22.5-\i*45:1.91) node {\footnotesize \x};}
					
					\foreach \x [count=\i] in {1001,0011,0111,1111,1110,1101,1010,0100}
					{\draw (-22.5-\i*45:3.41) node {\footnotesize \x};}
					
					\directionarrow
					
				\end{tikzpicture}\\
				\caption{The De Bruijn cycle $w_2 = (0010110100111100)$ as a lift of $u = (001\diam110\diam)$; $S(w_2)$ is shown as a subgraph of $S(u)$. The dotted edges are the edges of $S(u)$ not contained $S(w_2)$.}
				\label{fig: circle1 lift2}
			\end{center}
		\end{figure}
	\end{example}
	
	The drawings of $S(u)$ in Figures \ref{fig: circle1} and \ref{fig: circle2} suggest not only the existence of Hamilton cycles but also the existence of \emph{perfect factors} of De Bruijn graphs. Perfect factors in the De Bruijn graph have been studied, for example, in \cite{Mitchell97} and its references. Here, a perfect factor of a graph $G$ is a set of cycles of the same length whose vertex sets partition $V(G)$. 
	It is worth noting that for any upcycle $u$ for $\A^n$ with $n\mid a^{n-d}$ (including all $d=1$ upcycles), there exists a perfect factor consisting of $a^d$ cycles of length $a^{n-d}$. One construction is as follows. Start a walk at each vertex associated with some equivalence class as defined by $u$. From each of these vertices, follow the upcycle, repeating in place of diamond characters the same $d$ letters, in order, as appear in place of diamond characters in the label of the starting vertex. Since $n\mid a^{n-d}$, each walk will return to its starting vertex after the length of $u$, thus closing $a^d$ cycles of length $a^{n-d}$. For upcycles $u$ with diamondicity $1$, these cycles appear in $S(u)$ as concentric circles, as can be seen in Figures \ref{fig: circle1} and \ref{fig: circle2}.
	
	\subsection{Upcycles as generalized Euler tours}\label{subsec:Euler}
	
	A De Bruijn cycle with parameters $a$ and $n$ can alternatively be viewed as an Euler tour of the De Bruijn graph $B(a,n-1)$ where the words $x, y \in \A^n$ are covered consecutively in the De Bruijn cycle if and only if $x$ and $y$ are consecutive edges of the Euler tour. In fact, the line graph of the Euler tour representation of a De Bruijn cycle $u$ on $B(a,n-1)$ is the Hamilton cycle representation of $u$ on the graph $B(a,n)$. In the same way, one can consider a generalized Euler tour whose line graph is the generalized Hamilton cycle $S(u)$ defined in Section \ref{subsec:Hamilton}. For this generalization, it is useful to conceptualize an Euler tour as a set of ordered pairs of adjacent edges, with the pair $(x,y)$ indicating that edges $x$ and $y$ are consecutive in the tour.
	
	\begin{definition}\label{def:T(u)}
		Given an upcycle $u$ for $\A^n$, define $T(u)\subseteq\A^n\times\A^n$ to be a set of ordered pairs of edges of the De Bruijn graph $B(a,n-1)$ where $(x,y)\in T(u)$ if and only if $x$ and $y$ are adjacent in $B(a,n-1)$ and are covered consecutively in $u$.
	\end{definition}
	
	\begin{remark}
		The set $T(u)$ is precisely the edge set of $S(u)$.
	\end{remark}
	
	\begin{definition}
		Let $u$ be an upcycle for $\A^n$. We call $v \in B(a,n-1)$ a \emph{diamond vertex} corresponding to $u$ if $(x,y) \in T(u)$ for every pair of edges $x$ incident to $v$ and $y$ incident from $v$.
	\end{definition}
	
	The diamond vertices corresponding to $u$ are exactly the words covered by $u_i \cdots u_{i + n - 2}$ for some $i$ with $u_{i + n - 1} = \diam$. There are $\frac{d}{n}a^{n-d}$ diamonds in $u$ and $a^{d-1}$ total words covered by the $(n-1)$-length window preceding each, so there are $\frac{d}{n}a^{n-d}\cdot a^{d-1} = \frac{d}{n}a^{n-1}$ diamond vertices corresponding to $u$ in $B(a,n-1)$. For all other vertices $v$, each edge incident to $v$ is paired in $T(u)$ with exactly one edge incident from $v$.
	
	A drawing of $T(u)$ for $u=(001\diam110\diam)$ is shown on the right in Figure~\ref{fig:tour1}, with colored paths indicating consecutive edges through non-diamond vertices. Shown to the left is $S(u)$, highlighted to illustrate the equivalence of these models.
	
	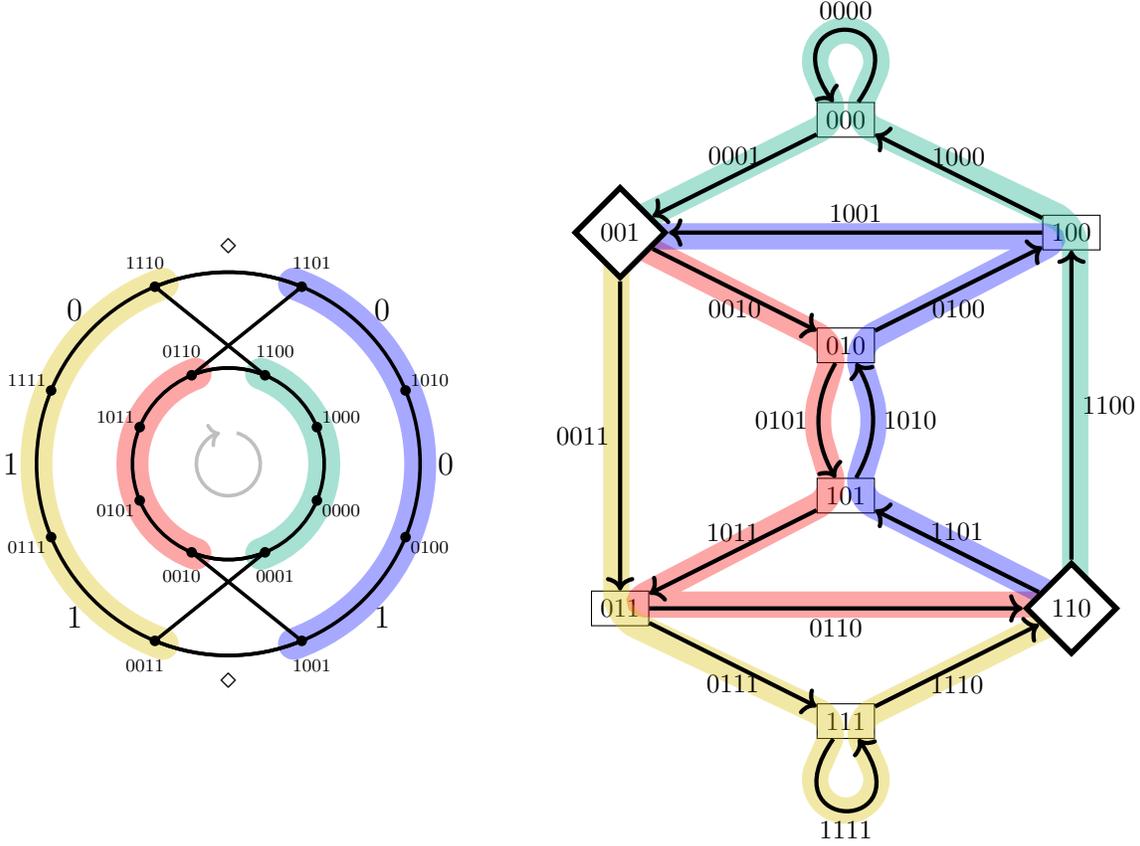
\begin{figure}[H]
		\begin{center}
			\begin{subfigure}[b]{0.4\linewidth}
				\scalebox{.85}{\begin{tikzpicture}
						{\draw[white] (-1,-6) -- (1,-6);   \draw[white] (-1,5) -- (1,5);}
						
						{\draw [0,opacity=.35,line width=5mm,line cap=round] (250:1.5) arc (250:110:1.5);}
						{\draw [1,opacity=.35,line width=5mm,line cap=round] (70:3) arc (70:-70:3);}
						{\draw [2,opacity=.35,line width=5mm,line cap=round] (250:3) arc (250:110:3);}
						{\draw [3,opacity=.35,line width=5mm,line cap=round] (70:1.5) arc (70:-70:1.5);}
						
						\foreach \x [count=\i] in {0,0,1,$\diam$,1,1,0,$\diam$}
						{
							{\foreach \y in {1.5,3}
								{\draw [ultra thick] (112.5-\i*45:\y) arc (112.5-\i*45:67.5-\i*45:\y);}
								{\draw [ultra thick] (90-\i*45:3.4) node {\Large{\x}};}}
						}
						
						\foreach \x in {1,2}
						{
							{\draw [ultra thick] (112.5+\x*180:1.5) arc (112.5+\x*180:67.5+\x*180:1.5);}
							{\draw [ultra thick] (112.5+\x*180:3) -- (67.5+\x*180:1.5);}
							
							{\draw [ultra thick] (112.5+\x*180:1.5) -- (67.5+\x*180:3);};
							{\draw [ultra thick] (112.5+\x*180:3) arc (112.5+\x*180:67.5+\x*180:3);}
						};
						
						\foreach \x in {1,...,8}
						{
							{\node[shape=circle,fill=black, scale=0.5] at (-22.5+\x*45:1.5) {};}
							{\node[shape=circle,fill=black, scale=0.5] at (-22.5+\x*45:3) {};}
						};
						
						\foreach \x [count=\i] in {0001,0010,0101,1011,0110,1100,1000,0000}
						{\draw (-22.5-\i*45:1.91) node {\footnotesize \x};}
						
						\foreach \x [count=\i] in {1001,0011,0111,1111,1110,1101,1010,0100}
						{\draw (-22.5-\i*45:3.41) node {\footnotesize \x};}
						
						\directionarrow
						
				\end{tikzpicture}}\\
				\caption{$S(u)$ for $u=(001\diam110\diam)$ with colored paths corresponding to colored paths in $T(u)$.}
			\end{subfigure}
			\hspace{.5cm}
			\begin{subfigure}[b]{0.5\linewidth}
				\scalebox{1}{\begin{tikzpicture}
						\node[draw] (000) at (0,0) {000};
						\node[draw] (001) at (-3,-1.5) {001};
						\node[draw] (010) at (0,-3) {010};
						\node[draw] (011) at (-3,-6.5) {011};
						\node[draw] (100) at (3,-1.5) {100};
						\node[draw] (101) at (0,-5) {101};
						\node[draw] (110) at (3,-6.5) {110};
						\node[draw] (111) at (0,-8) {111};
						
						{\color{0}
							\draw[opacity=.35,line width=3.5mm,line cap=round, rounded corners=3mm] (-3,-1.54) -- (-.12,-3) -- (-.42,-4) -- (-.12,-5) -- (-2.9,-6.45) -- (3,-6.45);}
						{\color{1}
							\draw[opacity=.35,line width=3.5mm,line cap=round, rounded corners=3mm] (3,-6.46) -- (.12,-5) -- (.42,-4) -- (.12,-3) -- (2.9,-1.55) -- (-3,-1.55);}
						{\color{2}
							\draw[opacity=.35,line width=3.5mm,line cap=round, rounded corners=3mm] (-3.05,-1.5) -- (-3.05,-6.6) -- (-0.1,-8) -- (-.5,-8.9) -- (0,-9.3) -- (.5,-8.9) -- (0.1,-8) -- (3.1,-6.5);}
						{\color{3}
							\draw[opacity=.35,line width=3.5mm,line cap=round, rounded corners=3mm] (3.05,-6.5) -- (3.05,-1.4) -- (0.1,0) -- (.5,.9) -- (0,1.3) -- (-.5,.9) -- (-0.1,0) -- (-3.1,-1.5);}
						
						\node[draw,diamond,line width=.75mm,fill=white] (001) at (-3,-1.5) {001};
						\node[draw,diamond,line width=.75mm,fill=white] (110) at (3,-6.5) {110};
						
						{\draw[ultra thick,->] (001) -- node[midway,below]{0010} (010);
							\draw[ultra thick,->,bend right] (010) to node[left]{0101} (101);
							\draw[ultra thick,->] (101) -- node[midway,above]{1011} (011);
							\draw[ultra thick,->] (011) -- node[midway,below]{0110} (110);}
						{\draw[ultra thick,->] (110) -- node[midway,above]{1101} (101);
							\draw[ultra thick,->,bend right] (101) to node[right]{1010} (010);
							\draw[ultra thick,->] (010) -- node[midway,below]{0100} (100);
							\draw[ultra thick,->] (100) -- node[midway,above]{1001} (001);}
						{\draw[ultra thick,->] (001) -- node[midway,left]{0011} (011);
							\draw[ultra thick,->] (011) -- node[midway,below]{0111} (111);
							\draw[ultra thick,loop,->,out=235,in=305,looseness=12] (111) to node[below]{1111} (111);
							\draw[ultra thick,->] (111) -- node[midway,below]{1110} (110);}
						{\draw[ultra thick,->] (110) -- node[midway,right]{1100} (100);
							\draw[ultra thick,->] (100) -- node[midway,above]{1000} (000);
							\draw[ultra thick,loop,->,out=55,in=125,looseness=12] (000) to node[above]{0000} (000);
							\draw[ultra thick,->] (000) -- node[midway,above]{0001} (001);}
				\end{tikzpicture}}\\
				\caption{$T(u)$ for $u=(001\diam110\diam)$. Colored paths indicate consecutive edges through non-diamond vertices.}
			\end{subfigure}
			\caption{$S(u)$ and $T(u)$ for $u=(001\diam110\diam)$. Here $a=2$ and $n=4$.}
			\label{fig:tour1}
		\end{center}
	\end{figure}
	
	One can think of $T(u)$ as the union of all Euler tours associated with De Bruijn lifts of $u$, a claim we will justify with Lemmas \ref{lem:length(n+1)wordsinPN} and \ref{lem:into-outof-diamondvertices}. Diamond vertices are the points at which the Euler tours associated with distinct De Bruijn lifts of $u$ may differ based on the expression of the following diamond character when lifting $u$. 
	
	\begin{lemma}\label{lem:length(n+1)wordsinPN}
		For all $a>1$, $n>1$, and $t\ge n$, every word of length $n+1$ appears in some $(a,n,t)$-perfect necklace.
	\end{lemma}
	\begin{proof}
		Take any arbitrary word in $\A^{n+1}$, say $x_1x_2\cdots x_{n+1}$. As in the proof of Proposition \ref{thm:StartingWith0sExists}, we will show that there exists an Euler tour of the astute graph $G(a,n-1,t)$ which corresponds to an $(a,n,t)$-perfect necklace containing the word $x_1x_2\cdots x_{n+1}$. Since the girth of $G(a,n-1,t)$ is $t$, we consider the following cases.
		
		\emph{Case 1.} If $t>n+1$, any length $n+1$ walk must be a path. Thus, begin at any vertex of $G(a,n-1,t)$, say $(0^{n-1},0)$, and construct a walk $W$ of length $n+1$ by taking edges according to the word $x_1x_2\cdots x_{n+1}$, ending at the vertex labeled $(x_3x_4\cdots x_{n+1},n+1)$. Since $W$ is a path, omitting $W$ from the graph decreases by one the in- and out-degree of, respectively, the last and first vertices of $W$, while all other vertices of $W$ have in- and out-degree both decreased by one. Hence, the remaining graph has an Euler trail starting at the last vertex of $W$ and ending at the first, altogether forming an Euler tour of $G(a,n-1,t)$ when appended to $W$.
		
		\emph{Case 2.} If $t=n+1$, then any length $n+1$ walk either forms a path, in which case the degrees of all vertices of $W$ are affected as in Case 1, or else forms a cycle. Thus, one can again begin at any vertex of $G(a,n-1,t)$, say $(0^{n-1},0)$, and construct a walk $W$ of length $n+1$ by taking edges according to the word $x_1x_2\cdots x_{n+1}$, ending at the vertex labeled $(x_3x_4\cdots x_{n+1},0)$. If $W$ forms a cycle, then omitting it from the graph decreases by one both the in- and out-degree of all vertices of $W$. Hence, the remaining graph has an Euler tour beginning and ending at the start of $W$, once again forming an Euler tour of $G(a,n-1,t)$ when appended to $W$.
		
		\emph{Case 3.} If $t=n$, then it is possible for a walk of length $n+1$ to repeat at most a single edge. To avoid this, begin at any vertex of $G(a,n-1,t)$ whose first letter is not $x_2$ or whose second letter is not $x_3$, say $({x_2}'{x_3}'\cdots {x_n}',0)$, where ${x_i}'$ is any letter in $\A-\{x_i\}$. Construct a walk $W$ of length $n+1$ by taking edges according to the word $x_1x_2\cdots x_{n+1}$, ending at the vertex labeled $(x_3x_4\cdots x_{n+1},1)$.
		
		\emph{Subcase 3.1.} For $n>2$, this selection of the first and second letters of the starting vertex ensures that $W$ is a path: ${({x_2}'{x_3}'\cdots {x_n}',0) \ne (x_2x_3\cdots x_n,0)}$ and ${({x_3}'{x_4}'\cdots {x_n}'x_1,1) \ne (x_3x_4\cdots x_{n+1},1)}$. Hence, the degrees of all vertices of $W$ are affected as in Case 1.
		
		\emph{Subcase 3.2.} In the case $t=n=2$, starting at the vertex $({x_2}',0)$ ensures that $W$ is a path only as long as $x_1 \ne x_3$. For words of the form $x_1x_2x_1$, starting here will result in a walk whose second and last vertices are equal: $W$ is the walk $({x_2}',0) \rightarrow (x_1,1) \rightarrow (x_2,0) \rightarrow (x_1,1)$. Thus, omitting $W$ from the graph decreases by one the out-degree of the first vertex of $W$, decreases by two the in-degree and by one the out-degree of the last vertex of $W$, and decreases by one both the in- and out-degree of the other vertices of $W$. Hence, the remaining graph has an Euler trail starting at the last vertex of $W$ and ending at the first, altogether forming an Euler tour of $G(a,n-1,t)$ when appended to $W$.
		
		Therefore, for $a>1$, $n>1$, and $t\ge n$, every word of length $n+1$ appears in some $(a,n,t)$-perfect necklace.
	\end{proof}
	
	\begin{lemma}\label{lem:into-outof-diamondvertices}
		Let $u$ be an upcycle, and let $x \in \A^{n-1}$ be a diamond vertex corresponding to $u$. Then for every $y,z\in\A$, there exists some De Bruijn lift $w$ of $u$ such that $(yx,xz)\in T(w)$.
	\end{lemma}
	\begin{proof}
		If $x$ is a diamond vertex covered by $u_i \cdots u_{i + n -2}$, then $u_{i-1} = u_{i + n -1} = \diam$. Lemma \ref{lem:length(n+1)wordsinPN} implies that every word of length $d+1$ must appear in some $(a,d,\tfrac{d}{n}a^{n-d})$-perfect necklace. Thus, for any $y,z \in \A$, there exists a perfect necklace $v$ such that the $d+1$ letters of $yxz$ which are covered by diamonds in $u$ are seen in order in $v$. In particular, we may choose $v$ so that $yxz$ appears in the De Bruijn lift $w$ of $u$ associated with $v$, so that $(yx,xz)\in T(w)$.
	\end{proof}
	
	\begin{prop}\label{prop:T(u)=unionofT(w)}
		Let $u$ be any upcycle, and let $I_u$ be the set of all De Bruijn lifts of $u$. Then ${T(u) = \bigcup_{w\in I_u} T(w)}$.
	\end{prop}
	\begin{proof}
		Note that if $w$ is a lift of $u$, then $T(w) \subseteq T(u)$ by Proposition \ref{prop:liftequivalence}. So $\bigcup_{w\in I_u} T(w) \subseteq T(u)$. Suppose $x \in \A^{n-1}$ and $y,z \in \A$ so that $(yx, xz) \in T(u)$. If $x$ is a diamond vertex, then by Lemma \ref{lem:into-outof-diamondvertices}, there is some De Bruijn lift $w$ of $u$ so that $(yx, xz) \in T(w)$. If $x$ is not a diamond vertex covered by $u_i \cdots u_{i + n - 2}$, then $u_{i-1}, u_{i + n - 1} \neq \diam$ and $(yx, xz) \in T(w)$ for every lift $w$ of $u$.
	\end{proof}
	
	\begin{prop}
		Let $u$ and $u'$ be upcycles with diamondicity $d$ and $d'$ respectively, and let $I_u$ be the set of De Bruijn lifts of $u$. Then $u'$ is a lift of $u$ if and only if $T(u')=\bigcup_{w \in J} T(w)$ for some subset $J$ of $I_u$.
	\end{prop}
	\begin{proof}
		If $T(u')=\bigcup_{w \in J} T(w)$ for some $J \subseteq I_u$, then $T(u') \subseteq \bigcup_{w \in I_u} T(w)$, which is $T(u)$ by Proposition \ref{prop:T(u)=unionofT(w)}. Equivalently, $E(S(u')) \subseteq E(S(u))$, hence $u'$ is a lift of $u$ by Proposition \ref{prop:liftequivalence}.
		
		Since any De Bruijn lift of $u'$ is a De Bruijn lift of $u$, the converse follows from Proposition \ref{prop:T(u)=unionofT(w)}, where $J$ is the set of De Bruijn lifts of $u'$.
	\end{proof}

	\begin{example}\label{ex: liftsEulerian}
		In Example \ref{ex: lifts}, we saw that the upcycle $u = (001\diam110\diam)$ lifts to just two distinct De Bruijn cycles, $w_1 = (0010110000111101)$ and $w_2 = (0010110100111100)$. Proposition \ref{prop:T(u)=unionofT(w)} gives that ${T(u) = T(w_1) \bigcup T(w_2)}$. This can be seen in Figure \ref{fig:tour1twolifts}, as $T(w_1)$ and $T(w_2)$ differ at exactly the diamond vertices, based on the choices made when unfolding each diamond.
	\end{example}
	
	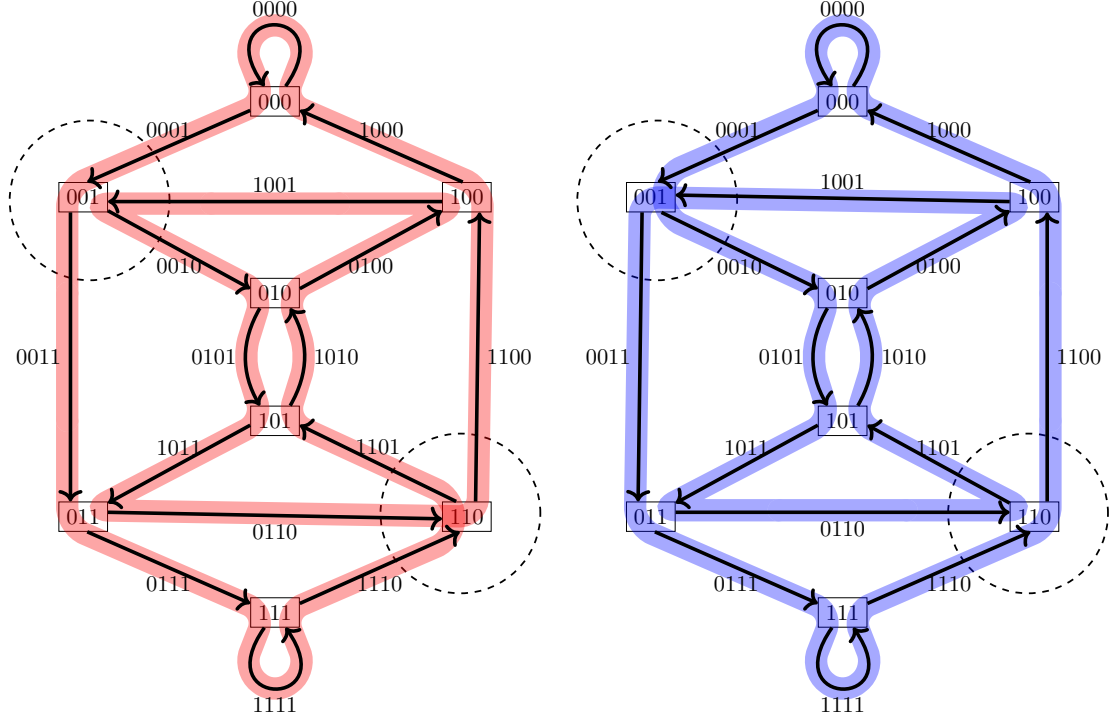
\begin{figure}[H]
		\begin{center}
			\begin{subfigure}[b]{0.45\linewidth}
				\scalebox{.85}{\begin{tikzpicture}
						\node[draw] (000) at (0,0) {000};
						\node[draw] (001) at (-3,-1.5) {001};
						\node[draw] (010) at (0,-3) {010};
						\node[draw] (011) at (-3,-6.5) {011};
						\node[draw] (100) at (3,-1.5) {100};
						\node[draw] (101) at (0,-5) {101};
						\node[draw] (110) at (3,-6.5) {110};
						\node[draw] (111) at (0,-8) {111};
						
						\draw[fill=none,dashed,thick](-2.9,-1.55) circle (1.25) node {};
						\draw[fill=none,dashed,thick](2.9,-6.45) circle (1.25) node {};
						
						{\color{0}
							\draw[opacity=.35,line width=3.5mm,line cap=round, rounded corners=3mm] (-1,-1.6) -- (-2.88,-1.6) -- (-.19,-3) -- (-.5,-4) -- (-.19,-5) -- (-2.88,-6.4) -- (3.17,-6.5) -- (3.25,-1.35) -- (0.15,0) -- (.5,.9) -- (0,1.3) -- (-.5,.9) -- (-0.15,0) -- (-3.25,-1.35) -- (-3.25,-3);
							
							\draw[opacity=.35,line width=3.5mm,line cap=round, rounded corners=3mm] (-3.25,-5) -- (-3.25,-6.65) -- (-0.15,-8) -- (-.5,-8.9) -- (0,-9.3) -- (.5,-8.9) -- (0.15,-8) -- (2.8,-6.85) -- (2.8,-6.2) -- (.19,-5) -- (.5,-4) -- (.19,-3) -- (2.88,-1.6) -- (1,-1.6);}
						
						{\color{white}
							\draw[line width=3.5mm,line cap=round, rounded corners=3mm] (1,-1.6) -- (-1,-1.6);
							\draw[line width=3.5mm,line cap=round, rounded corners=3mm]  (-3.25,-3) -- (-3.25,-5);
							
							\color{0}
							\draw[opacity=.35,line width=3.5mm,line cap=round, rounded corners=3mm] (1,-1.6) -- (-1,-1.6);
							\draw[opacity=.35,line width=3.5mm,line cap=round, rounded corners=3mm] (-3.25,-3) -- (-3.25,-5);}
						
						{\draw[ultra thick,->] (001) -- node[midway,below]{0010} (010.170);
							\draw[ultra thick,->,bend right] (010.-135) to node[left]{0101} (101.135);
							\draw[ultra thick,->] (101.-170) -- node[midway,above]{1011} (011);
							\draw[ultra thick,->] (011.10) -- node[midway,below]{0110} (110.185);}
						{\draw[ultra thick,->] (110.125) -- node[midway,above]{1101} (101.-10);
							\draw[ultra thick,->,bend right] (101.45) to node[right]{1010} (010.-45);
							\draw[ultra thick,->] (010.10) -- node[midway,below]{0100} (100);
							\draw[ultra thick,->] (100.-170) -- node[midway,above]{1001} (001.-10);}
						{\draw[ultra thick,->] (001.-130) -- node[midway,left]{0011} (011.130);
							\draw[ultra thick,->] (011.-75) -- node[midway,below]{0111} (111.160);
							\draw[ultra thick,loop,->,out=235,in=305,looseness=12] (111) to node[below]{1111} (111);
							\draw[ultra thick,->] (111.20) -- node[midway,below]{1110} (110.245);}
						{\draw[ultra thick,->] (110.60) -- node[midway,right]{1100} (100.-50);
							\draw[ultra thick,->] (100.105) -- node[midway,above]{1000} (000.-20);
							\draw[ultra thick,loop,->,out=55,in=125,looseness=12] (000) to node[above]{0000} (000);
							\draw[ultra thick,->] (000.200) -- node[midway,above]{0001} (001.75);}
				\end{tikzpicture}}\\
				\caption{$T(w_1)$ for ${w_1=(0010110000111101)}$.}
				\label{fig:tour1lift1}
			\end{subfigure}
			\begin{subfigure}[b]{0.45\linewidth}
				\scalebox{.85}{\begin{tikzpicture}
						\node[draw] (000) at (0,0) {000};
						\node[draw] (001) at (-3,-1.5) {001};
						\node[draw] (010) at (0,-3) {010};
						\node[draw] (011) at (-3,-6.5) {011};
						\node[draw] (100) at (3,-1.5) {100};
						\node[draw] (101) at (0,-5) {101};
						\node[draw] (110) at (3,-6.5) {110};
						\node[draw] (111) at (0,-8) {111};
						
						\draw[fill=none,dashed,thick](-2.9,-1.55) circle (1.25) node {};
						\draw[fill=none,dashed,thick](2.9,-6.45) circle (1.25) node {};
						
						{\color{1}
							\draw[opacity=.35,line width=3.5mm,line cap=round, rounded corners=3mm] (1,-6.4) -- (2.88,-6.4) -- (.19,-5) -- (.5,-4) -- (.19,-3) -- (2.88,-1.6) -- (-3.17,-1.5) -- (-3.25,-6.65) -- (-0.15,-8) -- (-.5,-8.9) -- (0,-9.3) -- (.5,-8.9) -- (0.15,-8) -- (3.25,-6.65) -- (3.25,-5);
							
							\draw[opacity=.35,line width=3.5mm,line cap=round, rounded corners=3mm] (3.25,-3) -- (3.25,-1.35) -- (0.15,0) -- (.5,.9) -- (0,1.3) -- (-.5,.9) -- (-0.15,0) -- (-2.8,-1.15) -- (-2.8,-1.8) -- (-.19,-3) -- (-.5,-4) -- (-.19,-5) -- (-2.88,-6.4) -- (-1,-6.4);}
						
						{\color{white}
							\draw[line width=3.5mm,line cap=round, rounded corners=3mm] (-1,-6.4) -- (1,-6.4);
							\draw[line width=3.5mm,line cap=round, rounded corners=3mm]  (3.25,-5) -- (3.25,-3);
							
							\color{1}
							\draw[opacity=.35,line width=3.5mm,line cap=round, rounded corners=3mm] (-1,-6.4) -- (1,-6.4);
							\draw[opacity=.35,line width=3.5mm,line cap=round, rounded corners=3mm] (3.25,-5) -- (3.25,-3);}
						
						{\draw[ultra thick,->] (001.-55) -- node[midway,below]{0010} (010.170);
							\draw[ultra thick,->,bend right] (010.-135) to node[left]{0101} (101.135);
							\draw[ultra thick,->] (101.190) -- node[midway,above]{1011} (011);
							\draw[ultra thick,->] (011.10) -- node[midway,below]{0110} (110.170);}
						{\draw[ultra thick,->] (110) -- node[midway,above]{1101} (101.-10);
							\draw[ultra thick,->,bend right] (101.45) to node[right]{1010} (010.-45);
							\draw[ultra thick,->] (010.10) -- node[midway,below]{0100} (100);
							\draw[ultra thick,->] (100.190) -- node[midway,above]{1001} (001.5);}
						{\draw[ultra thick,->] (001.-120) -- node[midway,left]{0011} (011.130);
							\draw[ultra thick,->] (011.-75) -- node[midway,below]{0111} (111.160);
							\draw[ultra thick,loop,->,out=235,in=305,looseness=12] (111) to node[below]{1111} (111);
							\draw[ultra thick,->] (111.20) -- node[midway,below]{1110} (110.255);}
						{\draw[ultra thick,->] (110.50) -- node[midway,right]{1100} (100.-50);
							\draw[ultra thick,->] (100.105) -- node[midway,above]{1000} (000.-20);
							\draw[ultra thick,loop,->,out=55,in=125,looseness=12] (000) to node[above]{0000} (000);
							\draw[ultra thick,->] (000.200) -- node[midway,above]{0001} (001.65);}
				\end{tikzpicture}}\\
				\caption{$T(w_2)$ for ${w_2=(0010110100111100)}$.}
				\label{fig:tour1lift2}
			\end{subfigure}
			\caption{Euler tours $T(w_1)$ and $T(w_2)$ for the two De Bruijn lifts of ${u=(001\diam110\diam)}$. Dashed circles highlight the differing consecutive paths through diamond vertices corresponding to $u$.}
			\label{fig:tour1twolifts}
		\end{center}
	\end{figure}

	\section{Pseudorandomness}\label{sec:pseudo}
	
	De Bruijn cycles and upcycles appear similar to randomly generated sequences. This notion is formalized by Golomb's Randomness Postulates: a randomly generated cyclic sequence of numbers from $\A = \set{0, 1, \ldots, a-1}$ is expected to satisfy the \emph{balance}, \emph{run}, and \emph{autocorrelation} properties \cite{GRPEnc}. Sequences with these properties have applications in cryptography and signal design \cite{GG05}. While these properties can be defined for cyclic sequences more generally, for simplicity we define them only for De Bruijn cycles. Given a De Bruijn cycle $w$ for $\mathcal{A}^n$, we define these properties as follows (with the necessary definitions for the R-3 property in Section \ref{subsec:autocor}).
	
	\begin{description}
		
		\item[Balance (R-1)] Each letter of $\A$ appears $a^{n-1}$ times in $w$.
		
		\item[Run (R-2)] For every $r \in [n]$ and every letter $\ell \in \A$, the word $\ell^r$ appears exactly $a^{n-r}$ times in $w$.

		\item[Autocorrelation (R-3)] 
		The modified De Bruijn cycle $\widehat{w}$ is a perfect autocorrelation sequence (Definition \ref{def:autocorrelation}). 
	\end{description}
	
	All De Bruijn cycles are known to satisfy the R-1 and R-2 properties. De Bruijn cycles that have been generated by a linear feedback shift register (LFSR) are known to satisfy the R-3 property (see, e.g., \cite[Section~II]{GK06}). 
	A longstanding conjecture of Golomb asserts that if $w$ is a De Bruijn cycle over $\A = \set{0,1}$ with the R-3 property, then $w$ can be generated by an LFSR \cite{Golomb80}. 
	
	In this section, we discuss these properties with upcycles in mind. In particular, in Section \ref{sec:uni_sub_dist}, we show that all upcycles have properties analogous to R-1 and R-2. Furthermore, in Section \ref{subsec:autocor}, we show that any De Bruijn cycle that is a lift of an upcycle with diamondicity $1$ cannot satisfy the R-3 property. This result shows that these De Bruijn cycles cannot be generated by LFSRs. 
	
	\subsection{Uniform subword distribution}\label{sec:uni_sub_dist}
	
	In this section, we introduce the \emph{partial subwords distribution property}, which generalizes the following well-known property of De Bruijn cycles (for example, see \cite[Proposition~5.2]{GG05}).
	
	\begin{prop}\label{prop:subwordproperty} Let $n \in \mathbb{N}$. For every $k \in [n]$, every De Bruijn cycle for $\A^n$ contains every word in $\A^k$ exactly $a^{n-k}$ times.
	\end{prop}
	
	\begin{proof}
		Let $w$ be a De Bruijn cycle for $\A^n$ and $k \in [n]$. Given a word $u$ of length $k$, there are exactly $a^{n-k}$ words of length $n$ that start with $u$, each of which appears exactly once in $w$.
	\end{proof}
	
	Leading up to the partial subwords distribution property, we first need the following two preliminary definitions. Recall that a partial word $u$ covers a word $w$ if there is some consecutive substring of $u$ of length $|w|$ for which the diamonds in this substring can be filled in to obtain $w$.
	
	\begin{definition}
		Given a cyclic partial word $u$ and a word $w$ of length at most $|u|$, define
		$$S_{\abs{w}}(u) = \set{u_iu_{i+1}\cdots u_{i+\abs{w}-1} : 1 \le i \le \abs{u}}$$
		to be the multiset of $\abs{w}$-windows of $u$.
		
		Then define 
		$$
		\mathcal{C}(u,w,q) = \set{v \in S_{\abs{w}}(u): v \text{ covers $w$ and $v$ contains exactly }q \text{ diamonds}}
		$$
		to be the multiset of $\abs{w}$-windows of $u$ that both cover $w$ and contain exactly $q$ diamonds
	\end{definition}
	
	\begin{definition}\label{def:prob_cover}
		If a cyclic partial word $u$ with exactly $q$ diamonds covers a word $w$ of the same length, then we say the \emph{expected multiplicity} of $w$ in $u$, denoted by $E(u,w)$, is $1/a^q$. If $u$ does not cover a word $w$, then we say $E(u,w) = 0$.
		
		If a word $w$ is shorter than a cyclic partial word $u$, 
		we say the \emph{expected multiplicity} of $w$ in $u$ is $E(u,w) = \sum_{v \in S_{\abs{w}}(u)} E(v,w)$, or more concretely:
		\begin{equation}\label{eq:expmultformula}E(u,w) = \sum_{q \ge 0}\abs{\mathcal{C}(u,w,q)}/a^q.\end{equation}
	\end{definition}
	
	\begin{remark}
		If $u$ is a partial cycle (or word), we may fill in the diamonds at random, where each replacement with $l \in \A$ occurs independently with probability $1/a$. For $w \in \A^k$, $E(u,w)$ is the expected number of times $w$ is covered by a randomly filled $u$.
	\end{remark}
	
	\begin{example} 
		Let $\mathcal{A}=\set{0,1}$, $u=(01\diam1\diam11011\diam11)$, and $w=011$. Then
		$$
		S_{\abs{w}}(u) = \set{01\diam,1\diam1,\diam1\diam,1\diam1,\diam11,110,101,011,11\diam,1\diam1,\diam11,110,101},
		$$
		so
		$$
		\mathcal{C}(u,w,q) = \begin{cases}
			\set{011} & \text{if } q=0 \\
			\set{01\diam,\diam11,\diam11} & \text{if } q=1 \\
			\set{\diam1\diam} & \text{if } q = 2 \\
			\varnothing & \text{otherwise}.
		\end{cases}
		$$
		Therefore $E(u,w)=1 + 3 \cdot 1/2 + 1/4 = 11/4.$
	\end{example}
	
	\begin{definition}\label{def:psdp}
		Suppose $u$ is a cyclic partial word over an alphabet of size $a$, and let $n \in \N$. We say $u$ has the \emph{partial subwords distribution property with respect to $n$} (for an upcycle $u$ for $\A^n$, simply \emph{partial subwords distribution property}) if, for every $k \in [n]$ and $v \in \A^k$, the expected multiplicity of $v$ in $u$ is $\abs{u}/a^k$.
	\end{definition}
	
	If $u$ is a cyclic partial word that has an $n$-periodic frame with diamondicity $d$ and length $a^{n-d}$, and $u$ has the partial subwords distribution property with respect to $n$, then taking $k=n$ in the above definition implies that $u$ is an upcycle for $\A^n$. We now prove the converse. 
	
	\begin{theorem}\label{thm:partial_subwords}
		Every upcycle has the partial subwords distribution property.
	\end{theorem}
	
	\begin{proof}
		Let $u$ be an upcycle for $\A^n$ with diamondicity $d$ and alphabet size $a$, and let $w$ be a De Bruijn cycle that is a lift of $u$, which exists by Corollary \ref{cor:lift_to_DB}. Let $k \in [n]$, and let $v \in \A^k$. Note that since $w$ has no diamonds, $E(w,v)$ counts the number of appearances of $v$ in $w$. By Proposition~\ref{prop:subwordproperty}, we know $E(w,v) = a^{n-k}$. We produce a second formula for $E(w,v)$ by partitioning the multiset of $k$-windows of $u$ which cover $v$ into the multisets $\mathcal{C}(u,v,q)$ for $q \ge 0$. We show that each element of $\mathcal{C}(u,v,q)$ corresponds to exactly $a^{d-q}$ appearances of $v$ in $w$ (and every appearance of $v$ in $w$ occurs in this way for some $q \ge 0$). 
		
		Let $s=u_i\cdots u_{i+k-1} \in \mathcal{C}(u,v,q)$, that is, a window of $u$ with length $|v|$ and exactly $q$ diamonds which covers $v$. There is a unique $n$-window $s'=u_i\cdots u_{i+n-1}$ of $u$ that begins at the same position as $s$. By the definition of lifting, $s'$ contributes $a^d$ length $n$ subwords in $w$, and $a^{d-q}$ of those begin with $v$. 
		We have $a^{n-k} = E(w,v) = \sum_{q \ge 0}a^{d-q}\abs{\mathcal{C}(u,v,q)}$. Dividing the left and right sides by $a^d$ and using Equation (\ref{eq:expmultformula}), we have $E(u,v) = \sum_{q \ge 0}\abs{\mathcal{C}(u,v,q)}/a^q = a^{n-k-d}$. 
	\end{proof}
	
	In Sections \ref{subsec:balance} and \ref{subsec:run}, we show that, as with the analogous properties for total words, the partial subwords distribution property implies partial words versions of the balance and run properties, which we call R-$1^{\diam}$ and R-$2^{\diam}$.
	
	\subsubsection{Balance}\label{subsec:balance}
	
	We define an analogue of the R-1 balance property, which we call the \emph{partial words balance property}.
	
	\begin{description}
		
		\item[Partial Words Balance (R-$1^{\diam}$)] 
		Given a partial word $u$ over an alphabet $\A$, we say $u$ has the R-$1^{\diam}$ property if there exists a $k\in\N$ such that each letter of $\A$ appears $k$ times in $u$.
	\end{description}
	
	Note that if a cyclic partial word $u$ with the R-$1^{\diam}$ property has length $N$ and has $D$ diamonds, then $k=(N-D)/a$. For an upcycle $u$ for $\A^n$ with diamondicity $d$, the R-$1^{\diam}$ property is equivalent to the statement that each letter of $\A$ appears $\frac{n-d}{n}a^{n-d-1}$ times in $u$. In particular, if $u$ is a De Bruijn cycle, then each letter appears $a^{n-1}$ times, so the R-$1^{\diam}$ property is the same as the R-1 property. In the following theorem, we show that Theorem \ref{thm:partial_subwords} implies that every upcycle has the R-$1^{\diam}$ property.

	\begin{theorem}\label{thm:equidistribution} Every upcycle has the partial words balance property R-$1^{\diam}$.
	\end{theorem}
	
	\begin{proof}
		Let $\ell \in \A.$ By Theorem \ref{thm:partial_subwords}, $u$ has the partial subwords distribution property, so $E(u,\ell)=a^{n-d-1}.$ There are exactly $\frac{d}{n}a^{n-d}$ diamonds in $u$, and each of them covers $\ell$ with expected multiplicity $1/a$. Each other contribution to the sum $E(u,w)$ comes from the letter $\ell$ appearing in $u$, so there are exactly
		$$
		a^{n-d-1} - \frac{d}{n} a^{n-d-1} = \frac{n-d}{n}a^{n-d-1}
		$$
		occurrences of $\ell$ in $u$.
	\end{proof}

	Theorem~\ref{thm:equidistribution} does not hold for non-cyclic upwords (which must be binary). For comparison, the binary upword $\diam^{n-1}01^n$ for $\set{0,1}^n$ has $n-1$ more $1$'s than $0$'s. For upcycles, however, Theorem~\ref{thm:equidistribution} implies number-theoretic constraints as the number of occurrences of each letter in an upcycle is an integer. 
	
	\begin{corollary}\label{cor:div}
		If there is an upcycle for $\A^n$ with diamondicity $d$ and $a := \abs{\A}$, then $n \mid d a^{n-d-1}$. In particular, if $\gcd(d,n)=1$, then $n \mid a^{n-d-1}$.
	\end{corollary}
	
	\begin{proof}
		The number of occurrences of $0$ in the upcycle is the integer $a^{n-d-1} - \frac{d}{n}a^{n-d-1}$ by Theorem~\ref{thm:equidistribution}.
	\end{proof}
	
	When $\gcd(d,n) = 1$, \cite[Theorem 4.11]{G18} implies that $n \mid a^{n-d}$. Thus Corollary~\ref{cor:div} is an improvement by a factor of $a$ in this case.
	
	\subsubsection{Run}\label{subsec:run}
	
	A \emph{run} of length $r$ in a (total) word is a sequence of $r$ consecutive letters of the same character. All De Bruijn cycles satisfy the run property R-2 \cite{Golomb80}. Some sources, such as \cite{GG05}, define runs as maximal, but we do not. 
	
	\begin{example}\label{ex:run0}
		The De Bruijn cycle $(0000100110101111)$ for $\set{0,1}^4$ contains the following numbers of runs of each length $r$ and letter $\ell$.
		
		\begin{center}
			\begin{tabular}{c|cccc}
				Run length: & $1$ & $2$ & $3$ & $4$\\
				\hline
				$\ell=0$ & $8$ & $4$ & $2$ & $1$\\
				$\ell=1$ & $8$ & $4$ & $2$ & $1$
			\end{tabular}
		\end{center}
	\end{example}
	
	To generalize the definition of the run property to accommodate partial words, we define the number of runs using expected multiplicity.
	
	\begin{definition}
		Suppose $u$ is a partial word over an alphabet $\A$, $\ell$ is a letter of $\A$, and $r$ and $q$ are integers. We define the \emph{number of $\ell^r$ runs} (or the \emph{number of length-$r$ runs of $\ell$'s}) in a partial word $u$ as $E(u, \ell^r)$.
	\end{definition}
	
	The number of $\ell^r$ runs in $u$ is not the number of occurrences of the string $\ell^r$, but the expected multiplicity of $\ell^r$ in $u$. Thus we may have non-integral numbers of runs.

	\begin{example}\label{ex:run1}
		The upword $\diam^{n-1}01^n$ for $\set{0,1}^n$ contains exactly $\frac{1}{2^{r-1}} + \frac{n-r}{2^r}$ length $r$ runs of $0$'s and exactly $n-r+1 + \frac{n-r}{2^r}$ length $r$ runs of $1$'s.
		
		For example, taking $n=4$, the upword $\diam\diam\diam01111$ has the following runs.
		
		\begin{center}
			\begin{tabular}{c|cccc}
				Run length $r$: & $1$ & $2$ & $3$ & $4$\\
				\hline
				$\ell=0$ & $5/2$ & $1$ & $3/8$ & $1/8$\\
				$\ell=1$ & $11/2$ & $7/2$ & $17/8$ & $1$
			\end{tabular}
		\end{center}
	\end{example}
	
	\begin{example}\label{ex:run2}
		The upcycle $(001\diam110\diam)$ for $\set{0,1}^4$ contains the following numbers of length-$r$ runs of $\ell$'s.
		
		\begin{center}
			\begin{tabular}{c|cccc}
				Run length $r$: & $1$ & $2$ & $3$ & $4$\\
				\hline
				$\ell=0$ & $4$ & $2$ & $1$ & $1/2$\\
				$\ell=1$ & $4$ & $2$ & $1$ & $1/2$
			\end{tabular}
		\end{center}
	\end{example}
	
	The tables in Examples \ref{ex:run1} and \ref{ex:run2} look quite different. The run lengths of the upcycle in Example~\ref{ex:run2} exhibit a similar pattern to the run property R-2, which we now state formally. Assume that $u$ is an upcycle for $\A^n$ with diamondicity $d$. The following is the \emph{partial words run property.} 
	
	\begin{description}
		
		\item[Partial Words Run (R-$2^{\diam}$)] For every $r \in [n]$ and every $\ell \in \A$, the number of $\ell^r$ runs in $w$ is $a^{n-d-r}$.
		
	\end{description}
	
	Substituting $d=0$ in the R-$2^{\diam}$ property yields the R-$2$ property. To compare the R-$1^{\diam}$ and R-$2^{\diam}$ properties, notice that we have defined a run in a partial word in such a way that the number of runs of length $1$ of a given letter $\ell$ is not equal to the number of occurrences of $\ell$, but to the number of occurrences of $\ell$ plus $1/a$ of the number of occurrences of $\diam$. In an upcycle, by Theorem~\ref{thm:equidistribution}, the number of runs of length $1$ of the letter $\ell$ is $\frac{n-d}{n}a^{n-d-1} + \frac{d}{n}a^{n-d}/a = a^{n-d-1}$, consistent with the partial words run property. 
	
	\begin{theorem}\label{thm:partial_run}
		Every upcycle has the partial words run property R-$2^{\diam}$.
	\end{theorem}
	
	\begin{proof}This theorem follows from Theorem \ref{thm:partial_subwords} by substituting $k=r$ and $v=\ell^r$ in Definition \ref{def:psdp}. 
	\end{proof}

	\subsection{Autocorrelation}\label{subsec:autocor}
	
	Autocorrelation measures how similar a cycle is to its cyclic shifts. For this section, we assume that $\A$ is the finite field $\mathbb{F}_{p^k}$ with $p^k$ elements, for $p$ a prime, since both autocorrelation and linear feedback shift registers apply only to this case.
	
	The following definition is useful both in constructing and analyzing De Bruijn cycles.
	
	\begin{definition}\label{def:mod/punc_DB}
		A \emph{modified} (or \emph{punctured}) \emph{De Bruijn cycle} for $\A^n$ is a cycle $\widehat{w}=(w_1 \cdots w_N)$ that covers every word in $\A^n\setminus\set{0^n}$ exactly once and does not cover $0^n$. 
	\end{definition}
	
	If $w=(w_1 \dots w_N)$ is a modified De Bruijn cycle for $\A^n$, then $N = a^n - 1$. The word $0^{n-1}$ appears exactly $a-1$ times in $w$, and adding a $0$ to any one of these instances produces a De Bruijn cycle. Given a De Bruijn cycle $w$, we often talk about its corresponding modified De Bruijn cycle $\widehat{w}$, which is created by deleting a $0$ from the word $0^n$ in $w$. In this case, we say $w$ is generated by an LFSR if $\widehat{w}$ can be generated by an LFSR.
	
	To discuss the R-3 autocorrelation property, we need the following definition, which describes the similarity between a cyclic word and its shift by some integer $\tau$.
	
	\begin{definition}[See Section~5.1.2 in \cite{GG05}]\label{def:autocorrelation}  Let $w=(w_1 \cdots w_N)$ be a cyclic word over $\A = \mathbb{F}_{p^k}$ and $\tau \in \mathbb{Z}$. Let $\xi = e^{2\pi \text{i}/ p}$, a primitive $p^{\mathrm{th}}$ root of unity. The \emph{autocorrelation at shift $\tau$} of $w$ is
		$$
		\displaystyle A(w, \tau)=\sum_{i=1}^{N}\xi^{Tr(w_{i+\tau} - w_i)}.
		$$
		where $Tr(x) = x + x^p + \cdots + x^{p^{k-1}}$, $x \in \mathbb{F}_{p^k}$ is the trace function from $\mathbb{F}_{p^k}$ to $\mathbb{F}_p$. If $A(w,\tau)=-1$ for all nontrivial shifts $\tau$ (that is, $\tau \not \equiv 0$ modulo $N$), then $w$ is called a \emph{perfect autocorrelation sequence}. If $w$ is a De Bruijn cycle and $\widehat{w}$ is a perfect autocorrelation sequence, then we say $w$ has the R-$3$ property.
	\end{definition}
	
	Notice that this definition is distinct from autocorrelation discussed in \cite{GO81}.
	
	\begin{prop}[Property~5.5 in \cite{GG05}]\label{prop:LFSRisR3}
		Let $w$ be a De Bruijn cycle for $\A = \mathbb{F}_{p^k}$. If $w$ is generated by an LFSR, then $\widehat{w}$ is a perfect autocorrelation sequence.
	\end{prop}

	Let $w=(w_1 \cdots w_N)$ be a cyclic sequence over $\A$ and let $\tau \in \mathbb{Z}$. The set of agreements at shift $\tau$, denoted $\mathbb{A}$, and the set of disagreements at shift $\tau$, denoted $\mathbb{D}$, are defined as follows:
	\begin{align*}
		\mathbb{A} &= \set{i \in [N] : w_{i+\tau} = w_i} \\
		\mathbb{D} &= \set{i \in [N] : w_{i+\tau} \ne w_i}
	\end{align*}
	The number of agreements and number of disagreements at shift $\tau$ are the respective cardinalities of $\mathbb{A}$ and $\mathbb{D}$.
	
	When $a=2$, we have $\xi = -1$, so $A(w, \tau)$ is the number of agreements minus the number of disagreements between $w$ and its cyclic shift by $\tau$. In fact, for non-binary upcycles, we can use a count on the number of agreements to show that a sequence is not a perfect autocorrelation sequence.
	
	\begin{lemma}\label{lem:agreements}
		Let $w = (w_1 \cdots w_N)$ be a cyclic word over $\A = \mathbb{F}_{p^k}$. If there is a nontrivial shift $\tau$ of $w$ that has at least as many agreements as disagreements with $w$, then $w$ does not have the R-3 property.
	\end{lemma}
	
	\begin{proof}
		Given $z \in \mathbb{C}$, let $\Re(z)$ be the real part of $z$. Assume $w$ and $\tau$ are as stated above. Then
		\begin{align*}
			\Re\Big(A\left(w,\tau\right)\!\Big) &= \sum_{i=1}^{N}\Re\left(\xi^{Tr(w_{i+\tau}-w_i)}\right) \\
			&= \sum_{i\in \mathbb{A}}\Re\left(\xi^{Tr(w_{i+\tau}-w_i)}\right) + \sum_{i\in \mathbb{D}}\Re\left(\xi^{Tr(w_{i+\tau}-w_i)}\right) \\
			&= \abs{\mathbb{A}} + \sum_{i \in \mathbb{D}}\Re\left(\xi^{Tr(w_{i+\tau}-w_i)}\right) \\
			&\ge \abs{\mathbb{A}} + \sum_{i \in \mathbb{D}}(-1) \\
			&= \abs{\mathbb{A}} - \abs{\mathbb{D}} \ge 0.
		\end{align*}
		Therefore $A(w,\tau) \ne -1$, so $w$ does not have the R-3 property.
	\end{proof}
	
	\begin{theorem}\label{thm:LiftNotR3}
		Let $u$ be an upcycle over $\A^n$ with $\A = \mathbb{F}_{p^k}$ and $d=1$. Then any De Bruijn lift of $u$ does not have the R-3 autocorrelation property.
	\end{theorem}
	
	\begin{proof}
		For any De Bruijn lift $w$ of $u$, we use Lemma~\ref{lem:agreements} with $\tau=-a^{n-1}$ (or equivalently $\tau = (a-1)a^{n-1}-1$) to show that the corresponding modified De Bruijn sequence $\widehat{w}$ does not have the R-3 property.
		
		Let $u = (u_1 u_2 \cdots u_{a^{n-1}})$ and assume $0^n$ is covered by $u_1\dots u_n$. Therefore $u_1 \ne \diam$ (otherwise $u_1\cdots u_{n+1} = \diam 0^{n-1} \diam$ would cover $0^n$ twice). Let $w = (w_1 w_2 \dots w_{a^n})$ be any lift of $u$ and assume $w_1\dots w_n = 0^n$. In other words, $w$ is a De Bruijn cycle with the following property: if $u_{i \bmod a^{n-1}} \ne \diam$, then $w_i = u_{i \bmod a^{n-1}}$. 
		Let $\widehat{w} = (w_2 \dots w_{a^n})$, i.e., $\widehat{w}$ is a modified lift of $u$. 
		We will find a lower bound on the percentage of agreements between $\widehat{w}$ and the shift of $\widehat{w}$ by $a^{n-1}$. If we can show that the lower bound is at least $1/2$, then we can apply Lemma~\ref{lem:agreements} to complete the proof. See Table~\ref{fig:compareItoshift} for a comparison of $\widehat{w}$ and its shift by $\tau = -a^{n-1}$.
		
		\begin{table}[h]
			\centering
			\begin{tikzpicture}[overlay]
				\fill[grey,opacity=.2] (-5.88,-1.3) -- (-3.76,-1.3) -- (-3.76,-0.3) -- (-5.88,-0.3) -- (-5.88,-1.3);
				\fill[grey,opacity=.2] (1.23,-1.3) -- (2.64,-1.3) -- (2.64,-0.3) -- (1.23,-0.3) -- (1.23,-1.3);
				\fill[grey,opacity=.2] (2.83,-1.3) -- (4.24,-1.3) -- (4.24,-0.3) -- (2.83,-0.3) -- (2.83,-1.3);
				\fill[grey,opacity=.2] (5.28,-1.3) -- (7.02,-1.3) -- (7.02,-0.3) -- (5.28,-0.3) -- (5.28,-1.3);
			\end{tikzpicture}
			$$
			\begin{array}{rlllllllll}
				\widehat{w}: & (\ w_2 & w_3 & \cdots & w_{a^{n-1}+1} & w_{a^{n-1}+2}  & w_{a^{n-1}+3} & \cdots & w_{a^{n\phantom{-1}}\phantom{(a-1)}}\ ) \\
				\text{Shift:} & (\ w_{(a-1)a^{n-1}+1} & w_{(a-1)a^{n-1}+2} & \cdots & w_{a^n} & w_2 & w_3 & \cdots & w_{(a-1)a^{n-1}}\ )
			\end{array}
			$$
			\caption{Comparing $\widehat{w}$ (first row) to its shift by $\tau=-a^{n-1}$ (second row) in the proof of Theorem~\ref{thm:LiftNotR3}. The highlighted columns correspond to the indices of $I$.}
			\label{fig:compareItoshift}
		\end{table}
		
		Assume $a \ge 3$. We consider the characters $w_i$ with indices in the set $I = \set{i: a^{n-1}+2 \le i \le a^n}\cup\set{2}$, whose cardinality is $(a-1)a^{n-1}$. Since $d=1$, we have $n | a^{n-1}$. Notice that $I\cup\set{1}\setminus\set{2}$ is a cyclically consecutive interval of indices of $w$. By the diamondicity of $u$, precisely $1-1/n$ of the indices $i \in I\cup\set{1}\setminus\set{2}$ have $u_{i \bmod a^{n-1}} \ne \diam$. Since $u_1 \ne \diam$, precisely $|I|(1-1/n) - 1$ of the indices $i \in I\setminus\set{2}$ have $u_{i \bmod a^{n-1}} \ne \diam$, which implies that $w_i = w_{i-{a^{n-1}}}$ as $i \equiv i-a^{n-1} \pmod{a^{n-1}}$, i.e. $w_i$ agrees with its shift backward by $a^{n-1}$. For $i=2$, observe that
		\[
		w_2 = 0 = w_1 = w_{(a-1)a^{n-1}+1},
		\]
		where the last equality is justified by the fact that $u_1 \ne \diam$ and the definition of lift. Thus, for at least $(1-1/n)(a-1)a^{n-1}$ indices $i \in I$, $w_i$ agrees with its shift backward by $a^{n-1}$.

		Recalling that the length of $\widehat{w}$ is $a^n-1$, we have shown that
		$$
		Q(a,n) = \frac{a^{n-1}(a-1)\parens{1-\frac{1}{n}}}{a^n-1}
		$$
		is a lower bound on the percentage of agreements between $\widehat{w}$ and its shift by $\tau=-a^{n-1}$.
		
		Observe that $Q(3,6) = 405/728$ and $Q(4,4) = 144/255$. Both of these values are larger than $1/2$, and $Q(a,n)$ increases if we increase $a$ or $n$ from these values (we can see this by differentiating with respect to either $a$ or $n$). 
		Recall that for any upcycle, $\gcd(a,n) \ne 1$ by \cite[Proposition~16]{CKMS17}, and if $a \ge 3$, then $n \ge 4$ by \cite[Proposition~5.1]{G18}. Thus, if an upcycle exists for some $a\ge 3$ and $n$, we know that either $a = 3$ and $n \ge 6$ or $a \ge 4$ and $n \ge 4$. (In fact, using Corollary \ref{cor:frame3} in the next section, it would have been sufficient to check the values of $Q(3,9)$ and $Q(4,4)$.) This establishes the desired result for $a\ge 3$.
		
		Now assume $a = 2$. The above considerations for $w_i$ where $i=2$ and $i \ge 2^{n-1}+2$ still hold, thus the expression $Q(2,n)$ given above is a lower bound on the percentage of agreements between $\widehat{w}$ and its shift by $\tau = 2^{n-1}$. However $Q(2,n)$ is not greater than $1/2$, so we must improve this lower bound. 
		
		We now consider the characters $w_i$ for $3 \le i \le 2^{n-1}+1$, i.e., the characters we had not considered in deducing $Q(a,n)$ above. Note that if $u_{i-1}u_{i} = \ell \ell$ for $\ell \in \set{0,1}$, then $w_i = w_{2^{n-1}+(i-1)}$ (recall that indices of $u$ are taken modulo $2^{n-1}$, so $u_{2^{n-1}+1}$ is $u_1$). In other words, for every total word $00$ or $11$ in the upcycle $u$ (except for at most one occurrence at $u_1u_2$, which is outside the interval of $i$ values under consideration), we get an agreement between $\widehat{w}$ and its shift by $\tau=-2^{n-1}$ (when considering the indices $3 \le i \le 2^{n-1}+1$).
		
		Thus we will count the number of times $00$ and $11$ appear as (total) subwords of $u$. 
		Recall that $E(u,v)$ is the expected multiplicity of $v$ in $u$, and $\mathcal{C}(u,v,q)$ is the multiset of consecutive substrings of $u$ of length $\abs{v}$ that cover $v$ and contain exactly $q$ diamonds. By Theorem \ref{thm:partial_subwords}, $u$ has the partial subwords distribution property, so
		\begin{equation}
			E(u,00) + E(u,11) = 2^{n-3} + 2^{n-3} = 2^{n-2}.
		\end{equation}
		Expressed differently, by Equation (\ref{eq:expmultformula}), we have
		\begin{align}
			E(u,00) &= \abs{\mathcal{C}(u,00,0)} + \frac{\abs{\mathcal{C}(u,00,1)}}{2} \\
			E(u,11) &= \abs{\mathcal{C}(u,11,0)} + \frac{\abs{\mathcal{C}(u,11,1)}}{2}
		\end{align}
		Observe that the number of times that $00$ and $11$ appear as (total) subwords of $u$ is exactly the quantity $\abs{\mathcal{C}(u,00,0)}+\abs{\mathcal{C}(u,11,0)}.$ There are exactly $2^{n-1}/n$ diamonds in $u$, from which we can deduce that
		\begin{equation}
			\abs{\mathcal{C}(u,00,1)} + \abs{\mathcal{C}(u,11,1)} = 2 \cdot \frac{2^{n-1}}{n},
		\end{equation}
		since each diamond in $u$ adds two elements in total to these sets. 
		Combining the numbered equations from above, we get that
		\begin{align*}
			\abs{\mathcal{C}(u,00,0)}+\abs{\mathcal{C}(u,11,0)} &= E(u,00) + E(u,11) - \frac{\abs{\mathcal{C}(u,00,1)}+\abs{\mathcal{C}(u,11,1)}}{2} \\
			&= 2^{n-2} - \frac{2^{n-1}}{n} \\
			&= 2^{n-2} \parens{ \frac{n-2}{n} }.
		\end{align*}
		
		Thus we have found $2^{n-2} \parens{\frac{n-2}{n}}$ occurrences of $00$ or $11$ in $u$, which contribute at least $2^{n-2} \parens{\frac{n-2}{n}}$ - 1 additional indices $3 \le i \le 2^{n-1}+1$ for which $\widehat{w}$ and its shift by $\tau$ agree. 
		Adding these to our numerator of the original estimate $Q(2,n)$ for the percentage of agreements, we get $R(n)$, which is the following improved lower bound on this percentage: 
		$$
		R(n) = \frac{2^{n-1}\parens{1-\frac{1}{n}}+2^{n-2} \parens{\frac{n-2}{n}} - 1}{2^n-1}.
		$$
		From Corollary~\ref{cor:div}, we must have $n=2^k$ with $k\ge 2$ when $a=2$. By Example \ref{ex:a=2,n=4,uniqueupcycle} there is exactly one upcycle for $n=4$ up to symmetries, which is $(0\diam001\diam11)$. This upcycle has exactly two distinct lifts (see Examples \ref{ex: lifts} and \ref{ex: liftsEulerian}), and it is straightforward to verify that neither of these have the R-3 property. 
		So we may assume that $n \ge 8$ in this case. Observe that $R(8) = 53/85$, and $R(n)$ increases as $n$ increases. Therefore we have established the desired result for $a=2$.
	\end{proof}
	
	The following result can be proved directly using the definition of LFSRs, but we note that it follows immediately from the above theorem.
	
	\begin{corollary}\label{cor:noLFSRs}
		Let $u$ be an upcycle with $d=1$ for $\A^n$ where $\A = \mathbb{F}_{p^k}$. No De Bruijn lift of $u$ can be obtained by an LFSR.
	\end{corollary}
	\begin{proof}
		By Theorem~\ref{thm:LiftNotR3}, no De Bruijn lift of $u$ has the R-$3$ property. Thus by Proposition~\ref{prop:LFSRisR3}, no De Bruijn lift of $u$ can be generated by an LFSR.
	\end{proof}
	
	It is unclear what the most useful extension of the R-3 property to upcycles might be. One option would be to say that an upcycle has the R-$3^{\diam}$ property if some (or perhaps all) of its De Bruijn lifts have the R-3 property. Another would be to modify the definition of autocorrelation to sum over only the positions $i$ where both $w_i$ and $w_{i+\tau}$ are not diamonds. Yet another option could be to consider some sort of expected value of autocorrelation, similar to what is contained in the definitions of the R-$1^{\diam}$ and R-$2^{\diam}$ properties.
	
	There are other pseudorandomness measures including \emph{discrepancy} (see, e.g., \cite{CH10,GS22}). One might consider what an appropriate extension of discrepancy to partial words might be or determine the discrepancy of lifts of known upcycles.
	
	\section{Nonexistence}\label{sec:nonexist}

	In \cite{G18} the authors give several number-theoretic conditions restricting the existence of upcycles for certain values of $(a,n,d)$. In this section we extend these results, and in addition we show that for several triples of $(a,n,d)$ for small $n$, no upcycles exist.
	
	\begin{definition}\label{def:nonexistence}
		Given a window, let $f \in \set{\bull,\diam}^n$ be its frame. The \emph{pane frame} $p$ of the window, or of the window frame, is the shortest word such that $p^s = f$ for some integer $s$. A word $w$ has \emph{period} $p$ if $w_j = w_k$ for all $j \equiv k \pmod{p}$. The \emph{frame period} of an upcycle, a window, or a window frame, is the minimum period of the frame of the upcycle, window, or window frame, respectively. Equivalently, the frame period is the length of a pane frame of a window.
	\end{definition}
	
	Theorem 4.11 of \cite{G18} shows that the frame period $m$ divides $n$, so any window frame of $u$ has exactly $n/m$ identical pane frames. Note that when diamondicity equals $1$, the frame period must equal $n$.
	
	The diamondicity property of upcycles implies that each upcycle has a unique window frame up to cyclic shifts. With this terminology, we can write a version of Theorem 4.11 from \cite{G18} as follows.
	
	\begin{theorem}[Theorem~4.11 in {\cite{G18}}]\label{thm:olddiv}
		Let $u$ be an upcycle for $\A^n$ with $a = \abs{\A}$ and diamondicity $d$. Let $m$ be the frame period of a window of $u$. Then $m \mid \gcd(a^{n-d},n)$.
	\end{theorem}
	
	We introduce the concept of a curtained window, which we will use to prove new nonexistence results.
	
	\begin{definition}\label{def:drapery}
		Let $f \in \set{\bull, \diam}^n$. If there exists a $k \in [n]$ such that for each $i \in [k]$, either $f_i = \diam$ or $f_{n-k+i} = \diam$, then $f$ is \emph{$k$-curtained}. If $f$ is $k$-curtained for some $k \in [n]$, then we say $f$ is \emph{curtained}. Similarly, a partial word is \emph{($k$-)curtained} if its frame is ($k$-)curtained.
	\end{definition}
	
	For example, the frame $\bull \diam \bull \bull \diam \bull$ is $2$-curtained, the frame $\bull \diam \diam \bull \bull$ is $3$-curtained, and the frame $\bull \diam \diam \bull \bull \bull$ is not curtained. In particular, any frame that begins or ends with a $\diam$ character is $1$-curtained. 
	
	In Theorem \ref{thm:windowframedrapery}, we show that some window frame of a nontrivial upcycle must not be curtained (in particular the frame of the window covering $0^n$ is not curtained). This is the main result that allows us to prove several nonexistence results for upcycles, and to provide upper bounds on the diamondicity of upcycles throughout this section.
	
	\begin{theorem}\label{thm:windowframedrapery}
		If a nontrivial upcycle has window frame $f$, then some cyclic shift of $f$ is not curtained. In particular, for any $\ell \in \A$, the window frame $f'$ covering the constant word $\ell^n$ is not curtained. 
	\end{theorem}
	
	\begin{proof}
		Let $u$ be a nontrivial upcycle for $\A^n$. Without loss of generality assume $u_1 \dots u_n$ covers $0^n$. We prove by contradiction that the window frame $f'$ covering $0^n$ is not curtained. In particular, we assume that $f'$ is $k$-curtained for some $k\leq n$.
		
		Consider the partial word $q=q_1\ldots q_n$ defined as follows.
		$$
		q_i = \begin{cases}
			0       & \text{if}~ u_i = 0 \\
			u_{i+k} & \text{if}~ u_i = \diam.
		\end{cases}
		$$
		It is immediate to see that $u_1 \dots u_n$ covers $q$. We now show that $u_{1+k} \dots u_{n+k}$ covers $q$. By Proposition~\ref{prop:upcycle_length}, these two windows are distinct, which contradicts that $u$ is an upcycle.
		
		Since $u_1\cdots u_n$ covers $0^n$, the first $n-k$ characters of $u_{1+k} \dots u_{n+k}$ cover $0^{n-k}$. Thus, for $1\leq i\leq n-k$, either $q_i=0$, which is covered by $u_{i+k}$, or $q_i=u_{i+k}=\diam=u_i$, so we have that the first $n-k$ characters of $q$ are covered by the first $n-k$ characters of $u_{1+k} \dots u_{n+k}$.
		
		We must now show that the last $k$ characters of $u_{1+k} \dots u_{n+k}$ cover the last $k$ characters of $q$. In other words, we need to show that $u_{n+j}$ covers $q_{n-k+j}$ for each $1\le j \le k$. By definition of $q$, notice that
		$$
		q_{n-k+j} = \begin{cases}
			0       & \text{if}~ u_{n-k+j} = 0 \\
			u_{n+j} & \text{if}~ u_{n-k+j} = \diam
		\end{cases}
		$$
		thus we only need to show that $q_{n-k+j}$ is covered by $u_{n+j}$ when $q_{n-k+j} = u_{n-k+j}=0$. Since the frame of $u_1 \dots u_n$ is $k$-curtained, if $u_{n-k+j}=0$ then $u_{j}=\diam$. Since the diamonds of $u$ are $n$-periodic, this implies that $u_{n+j}=\diam$. This shows that the remaining characters of $q$ are covered by the remaining characters of $u_{1+k} \dots u_{n+k}$. Thus any total word covered by $q$ is covered multiple times by the upcycle.
	\end{proof}
	
	\begin{example}\label{ex:a=2,n=4,uniqueupcycle}
		Up to reversal and complementation, $(001\diam110\diam)$ is the only upcycle for $\set{0,1}^4$.
	\end{example}
	
	\begin{proof}
		Suppose $u$ is an upcycle for $\set{0,1}^4$. By (the proof of) Theorem \ref{thm:windowframedrapery}, the window of $u$ that covers the word $0^4$ is not curtained, so it does not start or end with $\diam$, and it is not $\bull\diam\diam\bull$. Therefore the frame of the window of $u$ covering $0^4$ is $\bull\diam\bull\bull$ or its reverse $\bull\bull\diam\bull$. Then the length of $u$ is $2^{4-1} = 8$, and the window covering the word $1^4$ does not overlap with that covering $0^4$, so $u$ covers $0^41^4$ with the frame $\bull\diam\bull\bull\bull\diam\bull\bull$ or its reverse $\bull\bull\diam\bull\bull\bull\diam\bull$.
	\end{proof}
	
	\begin{lemma}\label{lem:remove_m}
		Let $p \in \set{\bull,\diam}^\ell$ and $f = p^s \in \set{\bull,\diam}^n$ for some positive integer $s$. For $k \in [\ell]$, $p$ is $k$-curtained if and only if $f$ is $k$-curtained.
	\end{lemma}
	
	\begin{proof}
		Let $p=p_1 \dots p_{\ell}$. As $p^s = f$, the first and last $\ell$ characters of $f$ are $p_1 \dots p_{\ell}$. Thus, for each $i \in [k]$ we have $f_i = p_i$ and $f_{n-k+i} = p_{\ell-k+i}$, so $p$ has a $k$-curtain if and only if $f$ has a $k$-curtain.
	\end{proof}

	\begin{corollary}\label{cor:windowpanedrapery}
		If a nontrivial upcycle has window frame $f$ and $f = p^s$ for some word $p$ and positive integer $s$, then some cyclic shift of $p$ is not curtained. In particular, some cyclic shift of the pane frame is not curtained.
	\end{corollary}
	
	\begin{proof}
		Let $w$ be an upcycle for $\A^n$, and let $f$ be a window frame of $w$ and $f=p^s$. Suppose for the sake of contradiction that every cyclic shift of $p$ is curtained. Then by Lemma~\ref{lem:remove_m}, every cyclic shift of $f$ is curtained, so, by Theorem \ref{thm:windowframedrapery}, $f$ is not a window frame of an upcycle.
	\end{proof}
	
	If $f\in \{\bull,\diam\}^n$ consists of all diamonds, then $f$ is $k$-curtained for all $k\in [n]$. It is clear that there is some minimum $d$ such that each $f \in \{\bull,\diam\}^n$ with at least $d$ diamonds is curtained. This idea can be used to provide a bound on the diamondicity of upcycles for $\A^n$, which motivates the following definition and subsequent results.
	
	\begin{definition}\label{def:D(n)}
		Define $D(n)$ to be the minimum value of $d$ such that every word in $\set{\bull,\diam}^n$ containing at least $d$ diamonds is curtained.
	\end{definition}
	
	\begin{lemma}\label{lem:diamdrapepane}
		Let $n$ and $s$ be positive integers. Every upcycle with a  window frame $f$ that can be written as $f = p^s$ for some word $p$, has diamondicity at most $(D(n/s)-1)s$. In particular, every upcycle has diamondicity at most $D(n)-1$.
	\end{lemma}
	
	\begin{proof}
		Let $w$ be an upcycle for $\A^n$ with frame period $m$, and let $f=p^s$ be a window frame of $w$. 
		By Corollary~\ref{cor:windowpanedrapery}, some cyclic shift of $p$ is not curtained. Hence, $p$ has fewer than $D(n/s)$ diamonds (since $p$ has length $n/s$), and hence $w$ has diamondicity at most $(D(n/s) - 1)s$.
	\end{proof}

	\begin{lemma}\label{lem:3letters}
		If $f \in \set{\bull, \diam}^m$ contains at most two cyclically consecutive $\bull$'s, then every cyclic shift of $f$ is curtained. Consequently, $f$ is not the pane frame or window frame of a nontrivial upcycle.
	\end{lemma}
	
	\begin{proof}
		If $f$ starts or ends with a $\diam$, then it is $1$-curtained. If $f$ starts with $\bull\diam$ and ends with $\diam\bull$, then it is $2$-curtained. If $f$ has at most two consecutive $\bull$'s, then every cyclic shift of $f$ falls into one of the above two cases and hence is $1$- or $2$-curtained. By Corollary~\ref{cor:windowpanedrapery} and Theorem \ref{thm:windowframedrapery}, $f$ is not the pane frame or the window frame of an upcycle.
	\end{proof}
	
	\begin{prop}\label{prop:drapebounds}
		$D(1) = D(2) = D(3) = 1$, and $D(4) = D(5) = 2$.
	\end{prop}
	
	\begin{proof}
		Let $m \in [3]$. If $f \in \set{\bull,\diam}^m$ has a diamond, then it has at most $2$ cyclically consecutive $\bull$'s. By Lemma~\ref{lem:3letters}, $f$ is curtained, and $D(m) = 1$.
		
		If $f \in \set{\bull, \diam}^4$ has $2$ diamonds, then it has at most $2$ cyclically consecutive $\bull$'s. Again, $f$ is curtained. Furthermore, $\bull \diam \bull \bull$ is not curtained, demonstrating that $D(4) = 2$.
		
		Suppose $f \in \set{\bull, \diam}^5$ has $2$ diamonds.  
		If $f$ starts or ends with a $\diam$, then it is $1$-curtained. If $f$ starts with $\bull\diam$ and ends with $\diam\bull$, then it is $2$-curtained. The remaining two cases are that $f$ is either $\bull \bull \diam \diam \bull$ or $\bull \diam \diam \bull \bull$, so $f$ is $3$-curtained. Thus $D(5) \leq 2$. Moreover, $\bull \diam \bull \bull \bull$ is not curtained, demonstrating that $D(5) = 2$.
	\end{proof}
	
	While we are able to determine $D(n)$ for small values of $n$ in Proposition \ref{prop:drapebounds}, we are unfortunately not able to do so in general. Rather, we find $D(n)$ computationally via an exhaustive search for $n\leq 39$ (see Table \ref{fig:D(n)Table}).
	
	\begin{remark}\label{asymptoticbound}
		The best asymptotic bound on $d$ at the time of writing is $d< n- \sqrt{n-\frac{7}{4}}-\frac{1}{2}$ by Proposition 4.7 in \cite{G18}.
	\end{remark}
	
	In \cite{G18}, it is shown that there are no nontrivial upcycles for $(a,n,d)$ if $\gcd(a^{n-d},n)\leq2$. In particular, this implies that there are no nontrivial upcycles with frame period 2. By a simple application of the previous results of this section, we are able to prove an extension of this result in Corollary \ref{cor:frame3} below. Additionally, we prove a result concerning the diamondicity of upcycles with frame period 4 or 5 in Corollary \ref{cor:4or5}.
	
	\begin{table}[h]
		\begin{center}
			\begin{tabular}{|c|c|} 
				\hline
				$\phantom{\Big|}n\phantom{\Big|}$ & $D(n)$ \\
				\hline
				1 & 1 \\
				2 & 1 \\
				3 & 1 \\
				4 & 2 \\
				5 & 2 \\
				6 & 3 \\
				7 & 4 \\
				8 & 4 \\
				9 & 5 \\
				10 & 6 \\
				11 & 6 \\
				12 & 7 \\
				13 & 8 \\
				\hline
			\end{tabular}\hspace{1cm}
			\begin{tabular}{|c|c|}
				\hline
				$\phantom{\Big|}n\phantom{\Big|}$ & $D(n)$ \\
				\hline
				14 & 9 \\
				15 & 9 \\  
				16 & 10 \\
				17 & 11 \\
				18 & 12 \\
				19 & 12 \\
				20 & 13 \\
				21 & 14 \\
				22 & 15 \\
				23 & 16 \\
				24 & 17 \\
				25 & 17 \\
				26 & 18 \\
				\hline
			\end{tabular}\hspace{1cm}
			\begin{tabular}{|c|c|}
				\hline
				$\phantom{\Big|}n\phantom{\Big|}$ & $D(n)$ \\
				\hline
				27 & 19 \\
				28 & 20 \\
				29 & 21 \\
				30 & 22 \\
				31 & 22 \\
				32 & 23 \\
				33 & 24 \\
				34 & 25 \\
				35 & 26 \\
				36 & 27 \\
				37 & 28 \\
				38 & 28 \\
				39 & 29 \\
				\hline
			\end{tabular}
		\end{center}
		\caption{Table of values for $D(n)$ for $n \le 39.$ The values for $n \le 5$ are shown in Proposition~\ref{prop:drapebounds}. The values for $n > 5$ were obtained via an exhaustive computer search. For any upcycle with parameters $(a,n,d)$, we must have $d < D(n)$ by Lemma~\ref{lem:diamdrapepane}.}
		\label{fig:D(n)Table}
	\end{table}
	
	\begin{corollary}\label{cor:frame3}
		Every nontrivial upcycle has frame period at least $4$. Consequently, there do not exist nontrivial upcycles for $(a,n,d)$ if $\gcd(a^{n-d},n)\leq3$.
	\end{corollary}
	
	\begin{proof}Suppose an upcycle $w$ with parameters $(a,n,d)$ has $\gcd(a^{n-d},n) \le 3$. By Theorem \ref{thm:olddiv}, the frame period $m$ of $w$ divides $\gcd(a^{n-d},n)$, so $m \le 3$. Proposition~\ref{prop:drapebounds} shows $D(m) = 1$, so by Lemma~\ref{lem:diamdrapepane}, $w$ has diamondicity $0$, and $w$ is a trivial upcycle.
	\end{proof}

	\begin{corollary}\label{cor:4or5}
		Let $m \in \set{4,5}$. Any nontrivial upcycle for $\A^n$ with frame period $m$ has diamondicity $\frac{n}{m}$. Every nontrivial upcycle for $\A^m$ has diamondicity $1$.
	\end{corollary}
	
	\begin{proof}
		Let $m \in \set{4,5}$. By Lemma~\ref{lem:diamdrapepane}, any nontrivial upcycle $w$ for $\A^n$ with frame period $m$ has diamondicity at most $(D(m) - 1)\frac{n}{m}$. By Proposition~\ref{prop:drapebounds}, $D(m) = 2$. Thus, its pane frame contains exactly one diamond, and $w$ has diamondicity $k$. Furthermore, by Theorem~\ref{thm:olddiv} and Corollary~\ref{cor:frame3}, any upcycle for $\A^m$ has frame period $m$.
	\end{proof}
	
	To conclude this section, we present a table that shows for $4\leq n\leq 12$, the possible values of $a$ and $d$ for which the results of this section do not rule out the possibility of the existence of an upcycle with parameters $(a,n,d)$ (see Table \ref{fig:ExistsTable}). This table serves to summarize what is known and to facilitate potential future research toward computational searches for upcycles with small parameters or more nonexistence results.
	
	Many of the restrictions on the values of $d$ in Table~\ref{fig:ExistsTable} come from Lemma~\ref{lem:diamdrapepane}. We also use Theorem~\ref{thm:olddiv} as well as Corollary~\ref{cor:frame3} to provide restrictions on $a$ for given $n$. For example, by Corollary~\ref{cor:frame3} there is no nontrivial upcycle with $n=6$ and $a=3$ since by Theorem~\ref{thm:olddiv} such an upcycle must have frame period at most 3. Corollary~\ref{cor:4or5} is also employed for example when $n=10$ and $a= 5k$ $(2 \nmid k)$, since such an upcycle must have frame period 5 by Theorem~\ref{thm:olddiv}, and a pane frame of this upcycle must have one diamond by Corollary~\ref{cor:4or5}.

	\begin{table}[h]
		\begin{center}
			\begin{tabular}{|c|c|c|}
				\hline
				${\mathbf{n}}$ & ${\mathbf{a}}$ & ${\mathbf{d}}$ \\
				\hhline{|=|=|=|}
				$4$ & $2k$ & $1$  \\
				\hline
				$5$ & $5k$ & $1$ \\
				\hline
				$6$ & $\begin{array}{c}
					6k \\
				\end{array}$ & $\begin{array}{c}
					1 \le d \le 2 \\
				\end{array}$  \\
				\hline
				$7$ & $7k$ & $1 \le d \le 3$  \\
				\hline
				$8$ & $2k$ & $1 \le d \le 3$  \\
				\hline
				$9$ & $3k$ & $1 \le d \le 4$  \\
				\hline
				$10$ & $\begin{array}{c}
					10k \\
					5k\ (2 \nmid k)
				\end{array}$ & $\begin{array}{c}
					1 \le d \le 5 \\
					2
				\end{array}$  \\
				\hline
				$11$ & $11k$ & $1 \le d \le 5$ \\
				\hline
				$12$ & $\begin{array}{c}
					12k \\
					2k\ (3 \nmid k) \\
				\end{array}$ & $\begin{array}{c}
					1 \le d \le 6 \\
					3 \\
				\end{array}$  \\
				\hline
			\end{tabular}
		\end{center}
		\caption{A table for $4\leq n\leq 12$ of the possible values of $a$ and $d$ for which the existence of an upcycle with parameters $(a,n,d)$ is not ruled out by 
			currently known results. Of these, only the cases where $n=4$, $a=2k$, $d=1$ and where $n=8$, $a=2k$, $d=1$ are known to exist.}\label{fig:ExistsTable}
	\end{table}

	\section{Open Questions}\label{sec:open}
	
	There is currently a large gap between the upcycles that are known to exist and the restrictions that show which upcycles cannot exist. Perhaps the most immediate open questions lie in shrinking this gap, either by constructing new upcycles or by providing new restrictions on their existence.
	
	On the question of existence, we may simply ask whether there exist upcycles for $n \ne 4,8.$ Given that all currently known upcycles can be constructed starting from binary upcycles with $d=1$, it may be profitable to focus on these upcycles in particular. Since $n \mid da^{n-d}$ for any upcycle, we know that all binary upcycles with diamondicity $d=1$ must have word length $n = 2^k$ for some $k \ge 2$. A natural question is whether binary upcycles with diamondicity $d=1$ exist for every greater power of $2$.
	
	\begin{conjecture}
		There exist upcycles for  $(a,n,d)=(2,2^k,1)$ for all $k\ge2$.
	\end{conjecture}
	
	To address this question it may be helpful to understand the binary upcycles for $n=8$ with $d=1$ better. Is there a method for constructing them? How many exist, up to symmetries?
	
	Though all known upcycles have $d=1$, the current best bound for $d$ in terms of $n$ is asymptotically equivalent to $n$ (see Remark~\ref{asymptoticbound}).
	
	\begin{question}\label{question:d>1}
		Do there exist upcycles with $d > 1$?
	\end{question}
	
	If $d=1$, then $n \mid a^{n-d}$. Thus if the answer to Question~\ref{question:d>1} is affirmative, we might ask whether there exist upcycles with $n \nmid a^{n-d}.$ Furthermore, in Theorems~\ref{thm: alphamult} (Alphabet Multiplier) and \ref{thm:necklacelift} (Characterization of lifts), the situation is more complicated if we do not assume $n \mid a^{n-d}$. The general proofs of these results rely heavily on perfect necklaces, so we ask the following.
	
	\begin{question}\label{question:perfneck}
		Can we find novel constructions of perfect necklaces? Are there efficient ways of generating them?
	\end{question}
	
	The best bounds on diamondicity for small values of $n$ rely on the function $D(n)$. Studying this function further may give better bounds on diamondicity, and we also believe that the definition of $D(n)$ is general enough to make the following question independently interesting.
	
	\begin{question}
		Is there a closed formula for $D(n)$?
	\end{question}
	
	Lastly we return to De Bruijn cycles. Corollary~\ref{cor:lift_to_DB} shows that every upcycle lifts to a De Bruijn cycle, so finding new upcycles is equivalent to finding these particular De Bruijn cycles. However, Corollary~\ref{cor:noLFSRs} shows that LFSRs (which are a particularly well-studied method of constructing De Bruijn cycles) cannot be used to create De Bruijn cycles that fold to any known upcycle.
	
	\begin{question}\label{question:foldableDB}
		Are there efficient ways of finding De Bruijn cycles that fold to upcycles?
	\end{question}

	\section*{Acknowledgments}
	
	The research of Dylan Fillmore, Rachel Kirsch, Kirin Martin, and Daniel McGinnis was partially supported by NSF DMS 1839918. Bennet Goeckner was supported by a grant from the Simons Foundation (Grant Number 814268, MSRI). Rachel Kirsch was supported in part by Simons Foundation Grant MP-TSM-00002688. Bennet Goeckner and Rachel Kirsch were also supported by AMS-Simons travel grants. Daniel McGinnis was supported by National Science Foundation (NSF) under award no. 2402145. We thank Bernard Lidick\'{y}, Anton Lukyanenko, Clare Sibley, and Elizabeth Sprangel for helpful discussions.

	\bibliographystyle{plainurl}
	\bibliography{UpcyclesBibliography}
	
	
	\vspace{.15in}
	
	{\sc Dylan Fillmore} Department of Mathematics, Iowa State University
	
	\vspace{.5em}
	
	{\sc Bennet Goeckner} Department of Mathematics, University of San Diego
	
	\vspace{.5em}
	
	{\sc Rachel Kirsch} Department of Mathematical Sciences, George Mason University
	
	\vspace{.5em}
	
	{\sc Kirin Martin} Department of Mathematics, Iowa State University
	
	\vspace{.5em}
	
	{\sc Daniel McGinnis} Department of Mathematics, Princeton University
\end{document}